\documentclass[11pt]{amsart}
\usepackage{amscd}
\usepackage{amsfonts}
 \usepackage{amssymb}
 \usepackage{amsmath}
\usepackage{color}
\usepackage{latexsym}

\usepackage{hyperref}

\numberwithin{equation}{section}

\newtheorem{pro}{{Proposition}}
\newtheorem{lemma}[pro]{{Lemma}}
\newtheorem{theorem}[pro]{{Theorem}}
\newtheorem{definition}[pro]{{Definition}}

\newtheorem{corollary}[pro]{{Corollary}}

\newtheorem{thm:intro}{Theorem}
\newtheorem{cor:intro}[thm:intro]{Corollary}
\newtheorem{def:intro}{Definition}

\newtheorem*{thm*}{Theorem}

\newcommand{\ra}{{\rightarrow}}
\newcommand{\lra}{{\longrightarrow}}

\newcommand{\cf}{{cf$.$\,}}

\newcommand{\cN}{{\mathcal{N}}}

\newcommand{\cA}{{\mathcal{A}}}

\newcommand{\cS}{S}  

\newcommand{\SU}{\mathop{\rm SU}\nolimits}

\newcommand{\al}{\alpha}
\newcommand{\be}{\beta}

\newcommand{\om}{\omega}

\newcommand{\we}{\wedge}
\newcommand{\lam}{\lambda}

\newcommand{\CC}{{\mathbb C}}

\newcommand{\HH}{{\mathbb H}}

\newcommand{\RR}{{\mathbb R}}

\newcommand{\PO}{\mathop{\rm PO}\nolimits}

\newcommand{\SO}{\mathop{\rm SO}\nolimits}

\newcommand{\Conf}{\mathop{\rm Conf}\nolimits}
\newcommand{\Aut}{\mathop{\rm Aut}\nolimits}

\newcommand{\Psh}{\operatorname*{Psh}\,}

\newcommand{\Diff}{\operatorname*{Diff}\,}
\newcommand{\Iso}{\mathop{\rm Iso}\nolimits}
\newcommand{\Isom}{\operatorname*{Iso}}  

\newcommand{\Hoth}{\mathop{\rm Hoth}\nolimits}

\newcommand{\sfD}{{\mathsf{D}}}

\newcommand{\C}{\mathcal{C}}
\newcommand{\Cs}{\mathcal{C}^{\infty}}

\newcommand{\res}{\mathrm{res}}
\newcommand{\Ind}{\mathrm{Ind}}
\newcommand{\Inds}{\Ind^{\infty} \,}

\newcommand{\supp}{\mathop{\mathrm{supp}}} 
\newcommand{\zsp}{\{0\}}

\begin{document}
\baselineskip 13pt 
\pagestyle{myheadings} 
\thispagestyle{empty}
\setcounter{page}{1}

\title[]{A note on vanishing of equivariant differentiable cohomology of
proper actions  and  application to  CR-automorphism and
conformal groups}
\author[]{Oliver Baues}
\address{Department of Mathematics\\ 
University of Fribourg\\
Chemin du Mus\' ee 23\\
CH-1700 Fribourg, Switzerland}
\email{oliver.baues@unifr.ch}

\author[]{Yoshinobu Kamishima}
\address{Department of Mathematics, Josai University\\
Keyaki-dai 1-1, Sakado, Saitama 350-0295, Japan}
\email{kami@tmu.ac.jp}

\begin{abstract} We establish that for any proper action of a Lie group on  
a manifold the associated equivariant differentiable cohomology groups with
coefficients in modules of $\Cs$-functions vanish in all degrees except than zero. 
Furthermore let $G$ be a Lie group of $CR$-automorphisms of a strictly pseudo-convex 
$CR$-manifold $M$. We associate to $G$ a canonical class in the first differential cohomology of 
$G$ with coefficients in the $\Cs$-functions on $M$. This class is non-zero if and only if $G$
is essential in the sense that there does not exist a $CR$-compatible strictly pseudo-convex 
pseudo-Hermitian structure on $M$ which is preserved by $G$. We prove that 
a closed Lie subgroup $G$ of $CR$-automorphisms 
acts properly on $M$ if and only if its canonical  class vanishes. 
As a consequence of Schoen's  theorem,  it follows that 
for any strictly pseudo-convex $CR$-manifold $M$,  
there exists a compatible strictly pseudo-convex pseudo-Hermitian structure 
such that the CR-automorphism group for $M$ and the group of pseudo-Hermitian
transformations coincide, except for two kinds of spherical $CR$-manifolds. 
Similar results hold for conformal Riemannian and K\"ahler manifolds.
\end{abstract}
\date{January 27, 2021} 
\thanks{This work was partially supported by JSPS grant No 18K03284}

\maketitle
\thispagestyle{empty}

\section{Introduction}

The original motivation for this article arises from the observation that the dynamics of certain groups of automorphisms of geometric structures, for  example in the cases of $CR$- and conformal manifolds, are to a large extent  controlled by a one-cocycle  for the \emph{differentiable cohomology} of those groups with coefficients in the $\Cs$-functions. In fact, as we will show,  automorphism groups of such geometric structures act  properly if and only if their  \emph{canonical  cohomology class}, which is defined by this  one-cocycle,  vanishes. 
\smallskip 

Behind the vanishing of the canonical class for proper actions in the above special cases there is a general vanishing principle for the equivariant cohomology groups of function spaces related to proper actions of Lie groups.  This vanishing principle  not only holds in degree one, but for arbitrary degree different from zero.  
The considerably more simple special case concerning only the first differentiable cohomology suffices for the original applications for $CR$ and conformal manifolds though. 

\smallskip
In this article, we will work out the details of the theory of \emph{equivariant} continuous and differentiable cohomology groups associated to actions of Lie groups on manifolds and we will prove the vanishing theorem for proper actions. This material is harking back and based on original work of Van Est, Mostow, Mostow-Hochschild \cite{VanEst, Mostow, HM}. 
Our exposition, 
as presented in Sections \ref{sec:cont_cohomology} and \ref{sec:smooth_coefficients},  will employ the language and techniques of sheaf cohomology  to develop 
the necessary steps of the proofs in a structured and concise way. \smallskip 

This introduction is organized as follows. Before diving into the  presentation of the theory of differentiable cohomology,  we start by discussing the aforementioned applications in the context of automorphism groups of \emph{$CR$- and conformal manifolds} in Section \ref{sec1.1}. The structure of proofs and the relation with differentiable cohomology will be discussed in Section \ref{sec1.2}. Finally Section \ref{sec1.3} gives a general  account on  equivariant differentiable cohomology groups associated to proper actions. 

\subsection{Automorphisms of $CR$ and conformal manifolds} \label{sec1.1}
Let $\omega$ be a contact form on a connected smooth manifold $M$. Assume further that there exists a complex structure $J$ on the contact bundle $\ker \om$ which is  compatible with $\om$ in the sense that the Levi form $d\om  \circ J$ is a positive definite Hermitian form. Then  $$(M,\{\om,J\})$$ is called a \emph{pseudo-Hermitian manifold}. 
Putting 
$\sfD = \ker\, \om$, 
the pseudo-Hermitian structure $\{\om,J\}$  induces  a \emph{$CR$-structure}  
$\{\sfD,J\}$ on $M$, which  
is then a \emph{strictly pseudo-convex} $CR$-structure.

\smallskip  
These structures naturally associate
to $M$ two important groups. 
Namely the group of pseudo-Hermitian transformations $$ \Psh\left(M,\{\om,J\}\right) 
\; , $$ 
and   the  group  of $CR$-automorphisms 
$$ \Aut_{CR}\left(M, \{\sfD,J\} \right) \, . 
$$
The isometry group of a Riemannian manifold $M$ and also any closed subgroup of the isometry group itself  are Lie groups,  by the theorem of Myers and Steenrod \cite{MS}. Moreover, the isometry group acts properly on $M$. The pseudo-Hermitian group $\Psh\left(M,\{\om,J\}\right)$ is a Lie group, since it preserves the associated contact Riemannian metric $$ g = \omega \cdot \omega + d\omega \circ J \, . $$ 
The automorphism group of a strictly pseudo-convex $CR$-manifold is also a Lie group which follows by \cite{Tanaka, WE}. 

\smallskip 
\paragraph{\em Strictly pseudo-convex $CR$-manifolds}

For any strictly pseudo-convex $CR$-manifold 
 $(M, \{\sfD,J\})$, there is, in general,  
no canonical choice 
in the conformal class of one-forms 
$$[\om] = \, \{ \, f \cdot \om \mid f \in \C^\infty(M, \RR^{>0})\, \} $$  which are 
representing  $\sfD$. Remark that the pseudo-Hermitian group 
$$\Psh\left(M, \{ f \cdot \om, J \}\right) $$ 
is always contained in the group $\Aut_{CR}(M, \{\sfD,J\})$, but 
it may vary considerably with the choice of $f$. 
Moreover, the Lie groups $\Psh(M,\{f \cdot \om,J\})$ act properly on $M$, whereas in special cases of certain spherical $CR$-manifolds    
$$\Aut_{CR}(M, \{\sfD,J\})$$ is too large and cannot act properly on $M$ (compare Theorem \ref{thm:schoen} below).

\smallskip 
Now let $M$ be a strictly pseudo-convex $CR$-manifold. Assuming that  
$\Aut_{CR}(M)$ is acting properly on $M$,  we shall prove
that 
there exists a  pseudo-Hermitian structure on $M$ compatible with its $CR$-structure 
such that 
$ \Psh(M)$ and $\Aut_{CR}(M)$ coincide. This fact is already mentioned in the literature 
(see for example the remark following Conjecture 1.4 in \cite{JL}) but it is 
hard to locate a concise proof. 

\smallskip  
In the light of the celebrated result of R.\ Schoen \cite{SC} on the properness of $CR$-automorphism groups  the following thus holds: 
 
\begin{thm:intro}\label{thm:main_CR}
Let $(M,\{\sfD,J\})$ be a strictly 
pseudo-convex $CR$-manifold.
Then either one of the following holds: 
\begin{enumerate}
\item[(i)] There exists a pseudo-Hermitian structure $\{\eta,J\}$, with $D= \ker \eta $, 
such that
\[ \Aut_{CR}(M,\{\sfD,J\})=\Psh(M,\{\eta,J\}) \; .\]
%
\item[(ii)]
 $M$ has a spherical $CR$-structure isomorphic to
either the standard sphere $S^{2n+1}$ or the Heisenberg Lie group $\cN = \cN_{2n+1}$.
\end{enumerate} In case (ii) the following holds: 
\[\left(\Psh(M),\Aut_{CR}(M)\right)=\begin{cases}
\left({\rm U}(n+1), {\rm PU}(n+1,1)\right) & (M=S^{2n+1}) \, ,\\
\left( \cN \rtimes {\rm U}(n), \cN\rtimes ({\rm U}(n)\times \RR^{>0}) \right) & (M=\cN_{}) \, .
\end{cases}\]
\end{thm:intro}
\medskip

For a \emph{compact} strictly 
pseudo-convex $CR$-manifold $M$, this result  is originally 
due to Webster \cite{WE}, see also \cite{KA1},  and  \cite[Proposition 4.4]{BGS} for a related result in the context of  compact Sasaki manifolds, compare also \cite{OK}. For further background on spherical $CR$-manifolds see \cite{BS,KA3}. 


%

\medskip 
\paragraph{\em Conformal Riemannian manifolds}
When we replace $\om$ by a Riemannian metric $g$ on $M$, 
there is a conformal analogue of Theorem \ref{thm:main_CR}.
For this recall that a diffeomorphism $\alpha:  M \to M$,  which 
satisfies $$ \alpha^*g \, = \, f \cdot g \;  , $$
for some positive function $f \in \Cs(M,\RR^{>0})$, 
is called a conformal automorphism. 
Let $ \Iso(M,g)$ denote the group of isometries and ${\rm Conf}(M,g)$ the group of conformal 
automorphisms for the metric $g$. The conformal  automorphism group is a Lie group, since the $G$-structure underlying conformal geometry has its second prolongation vanishing \cite{Kob}, but in general it doesn't act properly on $M$. 

\smallskip 
Similarly as in the $CR$-case, we will prove that if ${\rm Conf}(M,g)$ acts properly on $M$ then there exists a Riemannian metric $h$ conformal to $g$ such that ${\rm Conf}(M,g)=\Iso(M,h)$. In fact, this  was first observed long time ago by D.V.\  Alekseevsky, see \cite[proof of Theorem 1]{Al}. 

\smallskip 
Together with the classification of conformal Riemannian manifolds with non-proper automorphisms groups completed by the proofs of J.\ Ferrand \cite{Ferrand} and R.\ Schoen \cite{SC},  this gives rise to the following (see also \cite{Al}):

\begin{thm:intro}\label{thm:main_conformal}
Let $(M,g)$ be a Riemannian manifold.
Then either one of the following holds:
\begin{enumerate}
\item[(i)] There exists a Riemannian metric $h$
conformal to $g$ 
such that
\[ {\rm Conf}(M,g)=\Iso(M,h).\]
\item[(ii)]
 $M$ is conformal to 
either the standard sphere $S^{n}$ or the Euclidean
flat space $\RR^n$.
\end{enumerate}
For {\rm (ii)}, it occurs
$$  \left( \Iso(M),{\rm Conf}(M) \right)= \begin{cases}
\left( {\rm O}(n+1), {\rm PO}(n+1,1) \right) & (M=S^{n}) \, ,\\
\left(\RR^n\rtimes {\rm O}(n), \RR^n\rtimes ({\rm O}(n)\times \RR^{>0}) \right) & (M=\RR^n) \, .
\end{cases}
$$ 
\end{thm:intro}

\medskip 
\paragraph{\em Conformal K\"ahler  manifolds}
We shall  also be concerned with the 
 holomorphic conformal deformation of \emph{K\"ahler  manifolds}. 
Let $ (X,\{g,J\})$ be a K\"ahler manifold, where $g$ is a K\"ahler metric and $J$ denotes the complex structure on $X$. We let $\Iso\left(X,\{g,J \}\right)$ denote the associated  group of \emph{holomorphic} isometries and 
 $\mathrm{Conf}(X,\{g,J\})$ 
the  holomorphic conformal group. In this context, we obtain: 
\smallskip 


\begin{thm:intro} \label{thm:main_lcK}
Let $(X,\{g,J\})$ be a K\"ahler manifold, $\dim_{\RR} X = 2n \geq 4$.
Then either one of the following holds:
\begin{enumerate}
\item[(i)] 
There exists a Hermitian manifold  $(X, \{h, J\})$ with   the  Hermitian metric $h$  
conformal to the K\"ahler metric $g$ and 
$$   \mathrm{Conf}(X,\{g,J\}) \, =  \, \Iso\left(X,\{h,J \}\right) \,  . $$
\item[(ii)]   $X$ is holomorphically isometric to $\CC^n$. In this case,    
$$ \left( \Iso(X,\{g,J\}),{\rm Conf}(X,\{g,J\}) \right)  =
\left(\CC^n\rtimes {\rm U}(n), \CC^n\rtimes ({\rm U}(n)\times \RR^{>0})\right). $$

\end{enumerate} 
%
\end{thm:intro}

Note that Theorem \ref{thm:main_lcK}  gives 
a strong \emph{metric} rigidity property while in Theorem \ref{thm:main_CR} and Theorem \ref{thm:main_conformal} we have rigidity only up to a conformal map. 

\smallskip 
The Hermitian metric $h$ in Theorem \ref{thm:main_lcK} is globally conformal to a K\"ahler metric.
There is an important class of non-K\"ahler complex manifolds which carry Hermitian metrics  \emph{locally} conformal to K\"ahler metrics, see \cite{DO,OV,Va}, for example Hopf manifolds $S^{3} \times S^{1}$ fall into this class. We apply Theorem \ref{thm:main_lcK} to prove the existence of  {locally conformal K\"ahler metrics  whose holomorphic isometry group is maximal among all Hermitian metrics 
in a given conformal class,  see the discussion in Section \ref{sec:lcK} for further details. 

\subsection {Structure of proofs}  \label{sec1.2}

Theorem  \ref{thm:main_CR}, respectively Theorem \ref{thm:main_conformal},  are based mainly on two ingredients. One corner stone are the following celebrated rigidity results 
which were obtained in their  final form by R. Schoen \cite{SC} 
(see also J. Lee \cite{JL}) and J.\ Ferrand \cite{Ferrand}: 

\subsubsection{Schoen's rigidity theorem and related results}

\begin{thm:intro}\label{thm:schoen}
Let $(M,\{\sfD,J\})$ be a strictly 
pseudo-convex $CR$-manifold of dimension $2n+1$.
Then $\Aut_{CR}(M,\{\sfD,J\})$ acts properly on $M$ unless
 $M$ has a spherical $CR$-structure isomorphic to
either the standard sphere $S^{2n+1}$ or the Heisenberg Lie group $\cN_{2n+1}$.
\end{thm:intro} 

A result  analogous to Theorem \ref{thm:schoen} holds for conformal Riemannian manifolds, see \cite[Theorems 3.3, 3.4]{SC} and \cite{Ferrand}:  

\begin{thm:intro}\label{thm:schoen2}
Let $(M,g)$ be a Riemannian manifold.  
Then $\Conf(M,g)$ acts properly on $M$, unless
$M$ is conformally diffeomorphic to 
either the standard round  sphere $S^{n}$ or Euclidean flat space $\RR^{n}$.
\end{thm:intro} 

In the case of K\"ahler manifolds of complex dimension at least two, we strengthen 
the above conformal rigidity 
to the following  holomorphic metric rigidity:

\begin{thm:intro} \label{thm:schoenhol} 
Let $(X,\{g,J\})$ be a K\"ahler manifold, with $\dim X=2n\geq 4$. Then the holomorphic conformal group ${\Conf}(X,\{g,J\})$ acts properly on $X$, unless $X$ is holomorphically isometric to the complex space $\CC^n$. 
\end{thm:intro}

The proof of Theorem \ref{thm:schoenhol} can be found in Appendix  \ref{proofA2}. It combines Schoen's conformal rigidity with classical rigidity results on conformally flat K\"ahler manifolds (see Appendix \ref{A1}). 


\subsubsection{Differentiable cohomology of Lie groups associated to $CR$ and conformal actions} \label{subsec:cohomology_crc}
Schoen's theorems  setting the stage, the main additional ingredient used 
in the proofs of Theorem \ref{thm:main_CR} and Theorem \ref{thm:main_conformal} 
is the following vanishing result for  the \emph{first  
differentiable cohomology group} of Lie groups: 


\begin{thm:intro} \label{thm:main_h1}
Let $G$ be a Lie group that acts smoothly and properly on a  manifold $M$.  Then 
$$ H^1_{d}\left(G, \C^{\infty}(M, \RR)\right)  = \, \zsp \, . $$
\end{thm:intro} 

In Theorem \ref{thm:main_h1} the coefficients for cohomology are taken in the   \emph{differentiable $G$-module} of  
smooth functions  $\Cs(M, \RR)$,  which has  the
$\Cs$-topology of maps and with the natural $G$-action on functions. 
The equivariant differentiable cohomology group 
 $H^1_{d}\left(G, \C^{\infty}(M, \RR)\right)$  
is  then a  topological vector space, 
whose elements may be represented by certain smooth maps from $G$ to the locally convex vector space $\C^{\infty}(M, \RR)$. 

\smallskip
\paragraph{\em Canonical cohomology class associated to $CR$-actions} 
To each $CR$-action of a Lie group $G$ 
%
there exists  in the differentiable cohomology of 
$G$ 
a natural associated class 
$$ \mu_{\mathrm{CR}} =  [\lambda_{\mathrm{CR}}]  \; \in \, H^1_d\left(G, \C^\infty(M,\RR^{>0})\right)$$ 
which is induced by the $CR$-structure on $M$. The class  $\mu_{\mathrm{CR}}$ vanishes if and only if there exists a contact form $\eta$ compatible with the $CR$-structure,  such that $G$ is contained in the group of pseudo-Hermitian transformations $\Psh(M,\{\eta,J\})$. See Proposition \ref{pro:can_class_CR} in Section \ref{sec:cr} for further details. 

\smallskip 
Together with Theorem  \ref{thm:main_h1} this  reasoning reveals the following strong relationship between  
the above type of equivariant cohomology group and the properness of 
actions on $CR$- and conformal manifolds: 

\begin{thm:intro} \label{thm:main_cr_and_h1} Let $G$ be a closed  Lie subgroup of diffeomorphisms that  preserves either a strictly pseudo-convex $CR$-structure or a conformal Riemannian structure on $M$. 
Then $G$ acts properly on $M$ if and only if the first differentiable cohomology group
$H^1_{d}\left(G, \C^{\infty}(M, \RR)\right)$ vanishes.  
\end{thm:intro}

Conversely, taking into account that 
$\SU(n+1,1) = \Aut_{CR}(S^{2n+1})$ does not act properly on $S^{2n+1}$,  
and similarly for the other geometric group actions appearing in (ii) of Theorem \ref{thm:main_CR} and Theorem \ref{thm:main_conformal}, it follows that  these give examples of actions of semisimple Lie groups whose associated first equivariant differentiable cohomology groups are non-vanishing: 

\begin{cor:intro} \label{cor:main_cr_and_h1} The following hold: 
\begin{eqnarray*} & H^1_{d}\left(\SU(n+1,1), \C^{\infty}(S^{2n+1}, \RR)\right)\,  \neq  \; \zsp  \, , \\
&  H^1_{d}\left(\SO(n+1,1), \C^{\infty}(S^{n}, \RR)\right)\,  \neq  \; \zsp  \,  .  
\end{eqnarray*} 
\end{cor:intro}

The proof of Theorem \ref{thm:main_cr_and_h1} is given  in Section \ref{sec:proper_and_coho}. 
In fact,  Theorem \ref{thm:proper_and_coho}} which can be found there gives a much stronger characterization of proper actions which is also involving vanishing of higher cohomology. 

\subsection{Actions on manifolds and differentiable cohomology of Lie groups}

The methods which apply to prove Theorem \ref{thm:main_h1}, are based on 
foundational ideas on continuous and differentiable cohomology originally developed in the works of Mostow and Mostow-Hochschild \cite{Mostow,HM}, and also on somewhat later expositions  \cite{Blanc, CW}. In particular, we use the slice theory of proper actions, integration over compact groups, and Shapiro Lemma type results for equivariant differentiable cohomology of groups with coefficients in function spaces. Finally, the language of sheaf cohomology is applied to pursue the necessary patching of local data to a global vanishing theorem, see Section \ref{sec:cont_cohomology} and Section \ref{sec:smooth_coefficients} for details. 

\subsubsection{Vanishing of differentiable cohomology associated to proper actions} \label{sec1.3}
In the most general form presented in
Section \ref{sec:smooth_coefficients}, the vanishing theorem 
asserts, that, for any smooth and 
proper action of $G$ on $X$ and any differentiable $G$-module $V$,
the 
differentiable cohomology module $$H^{*}_{d}\left(G, \Cs(X, V)\right)$$
is acyclic, that is, $H^{r}_{d}\left(G, \Cs(X, V)\right) = \zsp$, $r \geq 1$. 
\smallskip

With respect
to the trivial $G$-module $V= \RR$,  this amounts to:   

\begin{thm:intro} \label{thm:main_cohomology}
Let $G$ be a Lie group which acts smoothly and properly on a manifold $X$. Then the 
differentiable cohomology groups of $G$ satisfy 
$$ H^{r}_{d} \left( G, \C^{\infty}(X, \RR) \right)  \, = \, \zsp\, \text{ , for all $r \geq 1$.} $$ 
\end{thm:intro}
The proof of Theorem \ref{thm:main_cohomology} and its further generalizations will be given in Section \ref{sec:smooth_coefficients}. 

\medskip
\paragraph{\em Further applications} 
We would like to point out that vanishing results of the above type are interesting in their own right and potentially bear many applications to the study of proper actions of Lie groups on geometric manifolds. For example,  note that we do not assume that $G$ is connected, so that the theorems  also apply to properly discontinuous actions and the covering theory of manifolds. If $G$ is discrete, the differentiable cohomology groups of $G$ with coefficients in a differentiable module $V$ reduce to the ordinary discrete cohomology  groups of $G$ with coefficients $V$. Note further  that, in case $G$ is acting properly discontinuously on $X$, Theorem \ref{thm:main_cohomology} is  known due to seminal work of Conner-Raymond \cite{CR}, see also \cite[Chapter 7.5]{LR}. 
The equivariant cohomology theory associated with properly discontinuous actions of groups and applications of the corresponding vanishing results to the topology of manifolds are discussed and surveyed in the  book \cite{LR}.

\section{Equivariant continuous cohomology} \label{sec:cont_cohomology}
In this section 
we start by reviewing  several facts about continuous cohomology of 
locally compact groups (compare  \cite{CW}, and also  \cite[Chapter IX]{BW}). 
Based on this we establish vanishing results 
for certain equivariant continuous cohomology groups 
related to proper actions on locally compact spaces. In the proof of the vanishing result
we take a sheaf theoretic approach to equivariant cohomology.

\medskip
\paragraph{\em Conventions} 
All spaces $X$ and $Y$ are assumed to be locally compact (and Hausdorff). We let $\C(X,Y)$ denote the space of continuous maps with the 
compact open topology.

\subsection{Continuous cohomology of locally compact groups}
Let $G$  be a locally compact topological group and $V$ a continuous $G$-module.
By definition, $V$ is a topological abelian group with a continuous action of $G$ 
by automorphisms. In addition we shall always assume that \emph{$V$ is a Hausdorff locally convex topological real vector space.} 
An isomorphism of continuous
 $G$-modules is a $G$-equivariant isomorphism
(linear homeomorphism) of topological vector spaces. 

\subsubsection{Definition of continuous cohomology groups} \label{sect:cc_groups}
Denote by $$ C^r(G;V) :=  \;  \C(G^r,V)$$
 the $G$-module of continuous (inhomogeneous) $r$-cochains of $G$ into $V$, 
which consists of continuous maps of the $r$-fold product $G^r=G\times\cdots\times G$
into $V$. 

The inhomogeneous coboundary operator
\begin{equation*}
\partial^r: C^r(G;V) \; \lra \;  C^{r+1}(G;V) \, ,
\end{equation*}satisfying $\partial^{r+1}\circ \partial^r=0$, 
is defined via:
\begin{equation}\label{eq:partial} \begin{split}
\partial^0 (v)(\al)&=\al\cdot  v-v\ \, (\al\in G, v\in V) \, , \\
\smallskip
\partial^r \lambda(\al_1,\dots,\al_{r+1})&=
\al_1 \cdot  \lambda (\al_2,\dots,\al_{r+1})\\
&\ \ \ \ +\sum_{i=1}^r(-1)^i  \lambda(\al_1,\dots,
\al_i\al_{i+1},\dots,\al_{r+1})\\
& \ \ \ \ +(-1)^{r+1} \lambda(\al_1,\dots,\al_{r}).
\end{split}\end{equation}
Define the \emph{continuous cohomology groups}  
$$ \displaystyle H^r(G,V) := \, {\rm ker}(\partial^r)\big/\,  {\rm im}(\partial^{r-1}) \, , \,r \geq 1$$ and
put $$ H^0(G,V) := \, {\rm ker}(\partial^0) = V^{G} \; . $$ 


\subsubsection{Cohomology of compact groups and integration} 

Recall that the locally convex vector space 
$V$ is called \emph{quasi-complete} if all 
closed  bounded subsets of $V$ are complete. 
For quasi-complete $V$ integration of compactly supported  continuous functions in $\C(G,V)$ is defined with respect to a left-invariant measure on $G$. 
That is, $V$ is a $G$-integrable module in the sense 
of \cite[\S 3]{HM} (see also \cite[2.13]{Mostow}). 

\smallskip
\paragraph{\em Remark} For a survey concerning the existence of vector valued integrals  refer to 
 \cite{Cass}. 
 
\begin{pro}[\mbox{
\cite[Lemma 7]{CW}}] \label{compsta}
Suppose that $G$ is compact and\/ let $V$ be a quasi-complete $G$-module.  
Then $$ \displaystyle H^r \! \left(G,V \right) \,= \,  \zsp  \, , \; r\geq 1\; .$$ 
\end{pro}
\begin{proof} 
Let $\lambda \in C^{r}\left(G; V \right)$ satisfy $\partial^{r} \lambda = 0$. With respect to a normalized finite Haar
measure $d\alpha$ on $G$ define  
$$ \tau (\alpha_1,\dots,\alpha_{r-1})  =  (-1)^{r} \int_{G} \lambda(\alpha_1,\dots,\alpha_{r-1},
\alpha)   \, d \alpha  \; , $$ to obtain $\tau \in C^{r-1}(G; V)$, which
satisfies
\[  \partial^{r-1} \tau =   \lambda \; . \qedhere \]

%
\end{proof}

%

\subsubsection{Shapiro's lemma} 
Let $H$ be a closed subgroup of $G$ and let $W$ be an $H$-module. Put 
$$ {\rm Ind}_H^G \, W = \; \{f\in \C(G,W)\mid f(g h^{-1})=h \cdot f(g), \ g\in G, h\in H\}.$$
Then $ {\rm Ind}_H^G \, W$ turns into a continuous $G$-module, by the action  
$$  (\alpha \cdot f) \left(g\right) =  f \left(\alpha^{-1} g \right) \; . $$

\smallskip \noindent
Remark that $G/H$ is paracompact \cite[Ch III \S6, Proposition 13]{bourbaki}. Suppose that the quotient map 
$$ G \, \to \, G \big/ H $$ 
admits continuous local cross sections. Then we have:

\begin{pro}[Shapiro Lemma, \cf  \mbox{
\cite[Propositions 3,4]{CW}}] \label{Shapi} 
$$ H^r \! \left(H,W\right) \; \cong  \; H^r \! \left(G, \, {\rm Ind \,}_H^G \, W \right) \, ,  \,  r\geq 0 .  $$ 
\end{pro}

\subsection{Equivariant cohomology groups} 
Let $X$ be a $G$-space and $V$ a continuous $G$-module. We consider the space of maps $$\C(X,V)$$  as  a continuous $G$-module,  where, for all $f \in  \C(X,V)$, $\alpha \in G$, we declare the action of $G$ as
\begin{equation}  \label{eq:induced_action}
 (\alpha \cdot f) \left(x\right) =  \alpha \cdot \left(f  \left(\alpha^{-1} \cdot x \right) \right) \; . 
\end{equation} 
We are interested in the  properties of the continuous cohomology of $G$ with coefficients in the $G$-module $\C(X,V)$.

\subsubsection{Associated bundle over $\, G/H$} Let $H$ be a closed subgroup of $G$. For any $H$-space $Y$ the diagonal action of $H$ on 
$G \times Y$ is   
\begin{equation} \label{eq:diagonal}
	h \cdot (g, y) = (g \, h^{-1}, \, h \cdot y) \, , \; \text{ for } h \in H.
\end{equation}
We then have the associated bundle with fiber $Y$ 
\begin{equation} \label{eq:abundle}
G \times_{H} Y \to \,  G/H \, . \end{equation} 
The space $G \times_{H} Y$ 
is defined by taking the  quotient 
of $G \times Y$ by the action \eqref{eq:diagonal}. 
Observe that $G$ acts on $G \times_{H} Y$ by left-multiplication on the first factor, which   
turns $G \times_{H} Y$ into a $G$-space and \eqref{eq:abundle} into a $G$-map. 

\smallskip 
\paragraph{\em Local cross sections} 
In order that \eqref{eq:abundle} is a fiber bundle we need to require that $G \to G/H$ is a fiber bundle: 
\begin{lemma} \label{lem:cfiberbundles} 
Suppose that $G/H$ admits local cross sections. Then:
\begin{enumerate}
\item The map $G \to G/H$ is a fiber bundle. 
\item The map $G \times_{H} Y \to \,  G/H$ is a fiber bundle.   
\item The quotient map  $G \times Y \to G \times_{H}  Y$ admits local cross sections. 
\end{enumerate}
\end{lemma}
\begin{proof} Let $s:U \to G$ be a local section for $\pi: G \to G/H$. Then the  map $g \mapsto (g H), s(gH )^{-1} g)$, $\pi^{-1}(U) \to U \times H$ is an $H$-equivariant homeomorphism.   Hence, $\pi$ is an $H$-principal bundle over $G/H$. 

In the view of (1), (2) is implied by \cite[II Theorem 2.4]{Bredon_compact_trans}. In fact, for any section $s$ of $\pi$, the map $U \times Y \to G \times_{H}  Y$, $(gH, y) \mapsto [s(gH), y]$ defines a local bundle chart for \eqref{eq:abundle}. 

Furthermore, in such bundle chart $(gH, y) \mapsto (s(gH), y)$, $U \times Y \to G \times Y$ defines a local cross section for $G \times Y \to G \times_{H}  Y $, showing (3). 
\end{proof}

Put  $ \C \left(G \times  Y,V\right)^{H} =  \{ f: G \times Y \to V \mid f(g  h^{-1}, h  \cdot y) =   f(g,y) \}$
for the 
subspace of $H$-invariant functions in $ \C \left(G \times  Y, V\right)$.  

\begin{lemma} \label{lem:cpushfw} 
Suppose that $G/H$ admits local cross sections. 
Then the  natural map 
$$  \iota: \:   \C \left(G \times  Y,V\right)^{H}  \to  \C \left(G \times_{H}  Y,V\right)  $$ 
is a homeomorphism.
\end{lemma}
\begin{proof}
Given $ \tilde f \in  \C \left(G \times  Y,V\right)^{H} \!$, $\iota(\tilde f) = f$
is the unique map satisfying  the relation 
$\tilde f = f \circ \mathsf{p}$, where $\mathsf{p}: G \times Y \to G \times_{H}  Y $ is the quotient map.
Now, for any  $C \subseteq  G \times_{H}  Y$ compact and $U \subseteq V$ open with $f(C) \subseteq U$, consider  $$\cN(C,U) = \{ \ell :  G \times_{H}  Y  \to V \mid \ell(C) \subseteq U \}$$ which is  an open neighborhood of $f$ in $\C \left(G \times_{H}  Y,V\right)$. By (3) of Lemma \ref{lem:cfiberbundles},  $\mathsf{p}:  G \times Y \to G \times_{H}  Y$ admits local cross sections. Since $G \times_{H}  Y$ is locally compact, we may 
cover $C$ with finitely many compact neighborhoods $C_{i}$ such that
$\mathsf{p}$ admits a section $s_{i}$ over $C_{i}$. Then put $\tilde C = \bigcup_{i} s_{i}(C_i \cap C)$. 
Hence,  $\tilde f \in  \cN(\tilde C, U)$ and,  for any $\tilde \ell  \in \cN(\tilde C, U) \cap  \C \left(G \times  Y,V\right)^{H}$, 
we have that $\iota(\tilde \ell) \in \cN(C,U)$. Hence, $\iota$ is continuous. 

Since $G \times_{H}  Y$ is locally compact Hausdorff, the inverse map
$$ \iota^{-1} \! : \; \C \left(G \times_{H}  Y,V\right)  \to  \C \left(G \times  Y,V\right)^{H}, \;  f \mapsto \tilde f = f \circ \pi$$
is clearly continuous (with respect to the compact open topology).
\end{proof}

\subsubsection{Adjointness properties of $\Ind_{H}^{G}$}
Let $V$ be a continuous $G$-module.
Since $H$ is a subgroup of $G$, the $G$-module $V$ is an $H$-module by restriction. The latter will be denoted  by $\res_H^G \, V$. 
Similarly, if $X$ is a $G$-space, we 
let $\res_{H}^{G} X$ denote the restricted $H$-space. 

\begin{lemma} \label{lem:ind} 
Let $V$ be a continuous $G$-module. 
There is a natural  isomorphism of continuous $G$-modules 
\begin{equation} \label{eq:ind}
\mu: \,  {\rm Ind}_{H}^{G} \, \C \left(Y,\res_H^G \,  V\right)  \, \to  \; \C \left(G \times_{H} Y,V\right) \;   .  
\end{equation}
Similarly, if $X$ is a $G$-space and $W$ an $H$-module, then  
there is a natural  isomorphism of continuous $G$-modules 
\begin{equation} \label{eq:ind2}  
\nu:  \C\left( X, \Ind^{G}_{H} \, W \right)  \, \to \,  \Ind^{G}_{H} \;   \C \left( \res_H^G \, X, W \right) \; . 
\end{equation}
\end{lemma}   

\begin{proof} 
Given $f \in {\rm Ind}_H^G\, \C(Y, \res_{H}^{G} V)  \subseteq \C(G, C(Y,V))$,  we define 
\begin{equation} \label{eq:corresp_maps}
\bar \mu(f)   \in  \C(G \times Y, V) \;  \text{,  via }
\bar \mu(f) (g,y) = g \cdot \left( \, f(g) \left(y\right) \, \right)   \; . \end{equation} 
Since $f$ is  $H$-equivariant  (by definition of  ${\rm Ind}_{H}^{G}$), it follows $$ \bar \mu(f)  \left(  h \cdot (g, y) \right) =   \; \bar \mu(f)  \left( g, y  \right) \; . $$
That is, $\bar \mu(f)$  is $H$-invariant for the diagonal action \eqref{eq:diagonal} and  descends to  
$$\mu(f)\in \C \left( G \times_{H}  Y, V  \right). $$ Conversely, every element of  $\C \left( G \times_{H}  Y, V  \right)$ arises in this way. 
One can verify easily that $\mu$ is $G$-equivariant, and also that it is a linear map. 
It remains to show that $\mu$ is a homeomorphism. 

Note first,  since $G$ is Hausdorff and $Y$ is locally compact, 
the natural map 
\begin{equation}   \label{eq:adjointc} \C(G, \C(Y,V)) \to \C(G \times Y, V) \end{equation}
is a homeomorphism of function spaces with respect to the compact open
topologies \cite[Ch X \S 3, Corollaire 2]{bourbaki}. 
It follows that $$ \bar \mu:  {\rm Ind}_H^G\, \C(Y,V) \to  \C \left(G \times  Y,V\right)^{H}$$ gives a homeomorphism. 
Finally, the push forward map $ \C \left(G \times  Y,V\right)^{H} \to \C \left( G \times_{H}  Y, V  \right)$ 
is a homeomorphism, by Lemma \ref{lem:cpushfw}. Therefore, $\mu$ is a homeomorphism.
 This proves \eqref{eq:ind}.
 
The proof of \eqref{eq:ind2} follows similarly, using the adjointness  \eqref{eq:adjointc}.
\end{proof}

\subsubsection{Slices, tubes  and induced representations} 
Let $X$ be a $G$-space and $p \in X$. Put $G_p =\{ g \in G \mid g \cdot p  =  p \}$.  
Let $\cS_{p}$ be a $G_{p}$-invariant locally closed subset of $X$ with 
$p \in \cS_{p}$. Then $\cS_{p}$ is called a \emph{slice} for the $G$-action on $X$ if $$ U_{p} = G \cdot \cS_{p}$$  is an open subset of $X$ and the map 
\begin{equation} \label{eq:tube}
 G \times_{G_{p}} \cS_{p} \,  \to \, U_{p} \; ,  \;   (g, s) \mapsto g \cdot s
\end{equation} 
is a homeomorphism. If a slice exists then $U_{p}$ is called a \emph{tube} around the orbit $G \cdot p$. 

For any \emph{$G$-invariant open subset} $ U \subseteq  U_{p}$, $p \in U$,  
we define  $$ \cS_{U} = U \cap \cS_{p} \, . $$ 
Then $\cS_U$ is a $G_p$-space. For any $G$-module $V$,  since $U$ is a $G$-space, $\C(U,V)$ is a $G$-module. 

\begin{lemma}[$G$-tubes] \label{lem:tubesc}
There is a natural  isomorphism of continuous $G$-modules 
\begin{equation} \label{eq:ind_tube}
{\rm Ind}_{G_{p}}^G \C(\cS_{U}, \, \res_{G_p}^G \! V)  \, \lra  \, \;  \C(U,V) \,   .  
\end{equation}
\end{lemma}

\begin{proof} Since $\cS_{p}$ is a slice at $p$, the 
natural map \eqref{eq:tube} is a  $G$-equivariant  homeomorphism, 
and so are the restricted maps \begin{equation} \label{eq:slice_homeo}
G \times_{G_{p}} \, \cS_{U} \to U \; .  
\end{equation}
In particular, $U$ is a $G$-tube around $G \cdot p$. In view of the homeomorphism \eqref{eq:slice_homeo}, the claim follows from Lemma \ref{lem:ind}, by taking $H = G_p$.
\end{proof}

We conclude: 

\begin{lemma}[Local vanishing] \label{lem:local_vanish}
Suppose there exists a slice $S_{p}$ for $p \in X$, such that $G_{p}$ 
is compact and $G \to G/G_{p}$  has local cross sections. Let $V$ be a quasi-complete continuous $G$-module.  
Then, for all sufficiently small 
$G$-invariant neighbourhoods $U$ of $p$,  
\begin{equation}\label{eq:smallvanish}
H^r \! \left(G,\C(U,V)\right)= \zsp  \, ,  \;  \,r\geq 1  \, .
\end{equation}
\end{lemma}

\begin{proof} 
By Lemma \ref{lem:tubesc}, $\C(U,V) =  {\rm Ind}_{G_{p}}^G  \C(\cS_{U}  , \res_{G_p}^G V)$. 
Shapiro's lemma (Proposition \ref{Shapi}) states that 
\begin{equation*}
 H^r \! \left( G, \, {\rm Ind}_{G_{p}}^G  \C(\cS_{U},\res_{G_p}^G V) \right) = \,  
H^r \! \left (G_p, \, \C(\cS_{U},\res_{G_p}^G V) \right) \; . 
\end{equation*}

As $V$ is quasi-complete 
also $\C(S_{U},V)$ is quasi-complete.
The module $\C(S_{U},V)$ is therefore $G_{p}$-integrable 
 (see in particular \cite[Proposition 3.1]{HM}). 
As $G_p$ is compact it follows from Proposition \ref{compsta} that 
 \[ 
\displaystyle H^r\!  \left(G_p, \, \C(\cS_{U},\res_{G_p}^G V) \right)= \zsp \, , \; r\geq 1.
\qedhere \] 
\end{proof}  

\subsection{Sheaf theoretic interpretation of equivariant cohomology} \label{sect:sheaves_c}
Let $X$ be a $G$-space and let $V$ be a continuous $G$-module. 
We let 
 $$\pi: X \to X/G$$  denote the quotient map for the $G$-action.  
%
Furthermore,  let
$$  \C_{X,V}  = \C( \cdot, V)$$  
denote the sheaf of continuous functions on $X$ with values in 
$V$,  as well as,  $$ \C_{X}  \text{ and } \C_{X/G}$$  the structure 
sheaves of continuous real valued functions on $X$ and $X/G$,  respectively. 
Since $G$ acts on $X$ and $V$, 
$\C_{X,V}$  is a $G$-sheaf. That is, the sheaf $\C_{X,V}$ has an action of $G$ by co-morphisms which is defined 
by  \eqref{eq:induced_action}. Remark further that  
$ \C_{X,V}$ is also a sheaf of  $\C_{X}$-modules. 
(For general background on sheaf theory, see eg.\  \cite[Chapter II]{Wells} 
or \cite{Bredon_sheaves}. For the notion of $G$-sheaves, see \cite[Chapitre V]{tohoku2}.) 

\smallskip 
\paragraph{\em Direct image sheaves} 
Let  $\cA$ denote any sheaf on $X$. The direct image $\pi_{*} \cA$ of $\cA$ is the sheaf on
$X/G$, where,  for any open subset $U$ of  $X/G$,
$$\pi_{*} \cA \left(U\right )  \, = \, \cA \left(\pi^{-1}(U)\right) \, . $$
We may view $\pi_{*}$ as a left-exact functor taking $G$-sheaves on $X$ to $G$-sheaves on 
$X/G$ (where $X/G$ has the trivial $G$-action). Of particular interest is the direct image 
sheaf of rings $\pi_{*} \,  \C_{X}$. Its subsheaf  $$\pi_{*}^{G} \,  \C_{X}$$  of
 $G$-invariant functions is called the  \emph{equivariant direct image} of $\C_{X}$. Note that 
 the sheaf $\pi_{*}^{G} \,  \C_{X}$ is canonically isomorphic to the sheaf $\C_{X/G}$
 of continuous functions on $X/G$. 

\smallskip 
\paragraph{\em Resolution of structure sheaves on $X/G$} 
We declare a differential sheaf 
\begin{equation} \label{eq:barresolution_c}
 C^{0}_{X,V}  \stackrel{\partial^{0}}{\lra} \, C^{1}_{X,V}   \stackrel{\partial^{1}}{\lra}\,  C^{2}_{X,V} \stackrel{\partial^{2}}{\lra}  \ldots  \; , 
\end{equation} 
of  $\pi_{*} \,  \C_{X}$-modules on $X/G$. The module of sections of  $C^{r}_{X,V}$ over an open subset $U \subseteq X/G$ is  defined as 
$$   C^{r}_{X,V}\left( U \right) =  \, C^{r} \! \left(G; \,  \C(\pi^{-1}(U), V) \right) \,  . $$ 
Here $\C(\pi^{-1}(U),V)$ is a $G$-module and,  by definition, $C^{r}_{X,V}\left( U \right)$  consists of the inhomogeneous $r$-cochains with coefficients  $\C(\pi^{-1}(U),V)$.  The differential 
$$  C^{r}_{X,V}\left( U \right) \,   \stackrel{\partial^{r}}{\lra}\,  C^{r+1}_{X,V}\left( U \right) $$ 
is obtained by the usual formula \eqref{eq:partial} for inhomogeneous cochains. It is trivial to verify that these local  maps 
patch together to define a homomorphism of sheaves $$ \partial^{r}: C^{r}_{X,V}  \, {\lra}\,   C^{r+1}_{X,V} \, . $$  

\smallskip
\paragraph{\em Equivariant direct image of\/ $\C_{X,V}$}
We remark  that  $C^{0}_{X,V} =  \pi_{*} \, \C_{X,V}$ is a sheaf of $G$-modules, 
since $ C^{0}_{X,V}(U) =  \C_{X,V}(\pi^{-1}(U))$.   
The subsheaf  
$$\pi_{*}^G \, \C_{X,V} $$
of \emph{$G$-invariant}  functions is generated by the presheaf
$$   U \, \mapsto  \; 
\C(\pi^{-1}(U),V)^{G}  \; . $$ Thus, we note that  
$$    \pi_{*}^G \, \C_{X,V}   = \ker \partial^{0} \; . $$ 

\smallskip
\paragraph{\em Remark} Observe that $\pi_{*}^G \, \C_{X,V}$ is not necessarily a sheaf of functions on $X/G$, unless the action of $G$ on $V$ is trivial or $G$ acts freely on $X$. (In case $V$ is a trivial $G$-module, we may identify $\pi_{*}^G \, \C_{X,V}$ with the structure sheaf $ \C_{X/G,V}$ of continuous functions on $X/G$ taking values in $V$.)

\medskip 
\paragraph{\em Local vanishing of continuous cohomology}
Suppose that, for all sufficiently small open neighbourhoods $U$ on $X/G$,    
we have 
\begin{equation}  \label{eq:local_v}
 H^r \left(G,\, \C(\pi^{-1}(U),V)\right)= \zsp \,  , \, \text{$r\geq 1$}  \, . 
  \end{equation}
Then we shall say that the  continuous cohomology groups of $G$ 
with coefficients $\pi_{*} \, \C_{X,V}$ \emph{vanish locally}.
This  is clearly equivalent to the condition that the sequence \eqref{eq:barresolution_c} 
is an exact sequence of sheaves: 
 

\smallskip

\begin{lemma} \label{lem:exact_c}
If the  continuous cohomology groups of $G$ with coefficients $\pi_{*} \, \C_{X,V} $ vanish locally  then 
\begin{equation} \label{eq:resolution}
\zsp  \, \lra  \,   \pi_{*}^G \, \C_{X,V}   \lra  \, C^{0}_{X,V}  \stackrel{\partial^{0}}{\lra} \, C^{1}_{X,V}   \stackrel{\partial^{1}}{\lra}\,  C^{2}_{X,V} \stackrel{\partial^{2}}{\lra}  \ldots  
\end{equation} 
is an exact sequence of sheaves  on $X/G$. 
\end{lemma} 

\paragraph{\em Cohomology of the  equivariant direct image sheaf $\pi_{*}^G \, \C_{X,V}$} 

By a standard argument on double complexes (compare \cite[Th\'eor\`eme 2.4.1]{tohoku1}), we can compute the sheaf cohomology groups $H^*_{X/G}(\pi_{*}^G \, \C_{X,V})$ as follows:

\begin{pro} \label{pro:ssequence}
Suppose that the  continuous cohomology groups of $G$ with coefficients $\pi_{*} \, \C_{X,V} $ vanish locally. Then there exists a spectral sequence converging to the sheaf cohomology $H^*_{X/G}(\pi_{*}^G \, \C_{X,V})$ with 
$$E^{p,q}_2 = H^p \left(H^q_{X/G}(C^*_{X,V})\right) . $$
\end{pro}

Note that the induced complex of global sections for the resolution  \eqref{eq:resolution} takes the form  
\begin{equation} \label{eq:sections_c}
   0 \lra  \,  \C(X,V)^{G}  \, \lra  \,  \C(X,V)    \stackrel{\partial^{0}}{\lra}\,  C^{1}(G;  \C(X,V) ) \stackrel{\partial^{1}}{\lra}  \ldots  \; . 
 \end{equation}
By construction this is the inhomogeneous bar complex for the continuous cohomology of $G$ with
coefficients in the $G$-module $ \C(X,V) $. Therefore 
$$ E^{p,0}_2 = H^p \!  \left(H^0_{X/G}(C^*_{X,V}) \right)= H^p \left(C^*(G;\C(X,V))\right) = H^p\left(G, \C(X ,V)\right). $$

\smallskip 
\paragraph{\em Paracompact quotient space $X/G$} 
Suppose that $X/G$ is a paracompact Hausdorff space. Then   \eqref{eq:resolution} gives a resolution of  $\pi_{*}^G \, \C_{X,V}$   by \emph{fine} sheaves:  
\begin{lemma} \label{lem:fine_c}
If  $X/G$ is a paracompact Hausdorff space, then the sheaves
\begin{enumerate}
\item $C^{r}_{X,V}$, $r \geq 0$,  and  
\item  $\pi_{*}^G \, \C_{X,V}$ 
\end{enumerate} 
are fine sheaves.
\end{lemma} 
\begin{proof} 
Since $X/G$ is a paracompact Hausdorff space, any locally finite covering 
by open sets admits a subordinate partition of unity.
This implies that  the structure sheaf $\C_{X/G}$ is a fine sheaf of rings.  
In particular,  any sheaf of $\C_{X/G}$-modules is a fine sheaf (see \cite[Theorem 9.16]{Bredon_sheaves}). Now $C^{r}_{X,V}$ is a sheaf of $\C_{X/G}$-modules, where, given $\epsilon \in \C_{X/G}$, an open subset $U \subseteq X/G$  and $ c \, \in \,   C^{r}_{X,V}\left( U \right) \, $,
we declare
$$  \epsilon \cdot c \; \left(g_{1}, \ldots, g_{r} \right)   =   (\epsilon  \circ \pi) \cdot c \left( g_{1}, \ldots, g_{r} \right) \; . $$
It follows that $C^{r}_{X,V}$ is a fine sheaf, hence (1). 

Since $\pi_{*}^G \, \C_{X,V}$ is a sheaf of $\pi_{*}^{G} \,  \C_{X} = \C_{X/G}$ -modules on $X/G$, it is a fine sheaf as well.
Thus (2) holds.
\end{proof}

In this situation  the equivariant continuous cohomology groups of $G$ may be expressed in terms of sheaf cohomology on $X/G$:

\begin{corollary} \label{cor:groupc=sheafc}
Suppose that $X/G$ is a paracompact Hausdorff space and that \eqref{eq:resolution} is exact.
Then, for all $r \geq 0$,  there is a natural isomorphism 
$$   H^{r} \left(G, \C(X,V) \right) \, \cong \, H^{r} \left( \pi_{*}^G \, \C_{X,V} \right)  
 \; .$$ 
\end{corollary} 
\begin{proof} The homology of the complex of global sections for 
any resolution of  $\pi_{*}^G \, \C_{X,V}$ by fine sheaves is  isomorphic to the 
sheaf cohomology  $H^{*} \left( \pi_{*}^G \, \C_{X,V} \right)$ (see, for example \cite[Section II.3]{Wells}).
By Lemma \ref{lem:fine_c} part (1), the resolution \eqref{eq:resolution} of $\pi_{*}^G \, \C_{X,V}$ is fine. Thus 
the cohomology of the sheaf $\pi_{*}^G \, \C_{X,V}$
is  isomorphic to the homology of the complex \eqref{eq:sections_c}. (In particular, the spectral sequence in Proposition \ref{pro:ssequence} collapses at $E_2$. That is $E_2^{p,q} = \zsp$, $q>0$.) 
\end{proof}

\subsection{Vanishing of equivariant continuous cohomology}
If $X/G$ is a paracompact Hausdorff space, then $\pi_{*}^G \, \C_{X,V}$ is a fine sheaf (by Lemma \ref{lem:fine_c} part (2)). Therefore 
its sheaf cohomology 
must be acyclic. In the view of Corollary \ref{cor:groupc=sheafc} this  
proves: 

\begin{theorem}[Vanishing theorem, continuous case] \label{thm:cont_vanish} \
Suppose that the 
continuous cohomology groups of $G$ with coefficients 
$\pi_* \, \C_{X,V}$ vanish locally and that $X/G$ is a paracompact Hausdorff space. Then  
$$H^r \left(G , \C(X,V) \right)= \zsp\,   \text{, for all $r \geq 1$} . $$ 
\end{theorem}

\smallskip
\paragraph{\em Proper actions of Lie groups}
A typical case for application arises in smooth actions on manifolds. Let\/ $G$ be a Lie group and let $V$ be a quasi-complete continuous $G$-module. Then we have:
\begin{corollary}[Vanishing theorem, smooth manifolds]  \label{cor:smooth_man} 
Let $G$ be a Lie group and $X$ a $G$-manifold on which $G$ acts smoothly and
properly. Then 
$$  H^{r}\left(G, \C(X,V) \right)   = 0 \;  , \text{ for all }  r  \geq 1  . $$
\end{corollary} 
\begin{proof} Since $X$ is paracompact and $G$ acts properly,  $X/G$ is a paracompact Hausdorff space. 
By the differentiable slice theorem (see \cite[\S 4, Lemma 4]{JK}, for example), every point 
$p \in X$ admits a  $G$-tube $U_{p}$. Moreover, since $G$ is a Lie group, for any closed subgroup 
$H$ of $G$, $G/H$ is a manifold and admits local (smooth) sections. Since $V$ is assumed quasi-complete, it  follows by Lemma \ref{lem:local_vanish} that the continuous cohomology of $G$ with coefficients $\C(X,V)$ vanishes 
locally. Therefore, Theorem \ref{thm:cont_vanish} applies.
\end{proof}

\section{Smooth coefficients and vanishing of equivariant differentiable cohomology}   \label{sec:smooth_coefficients} 
Let $G$ be a Lie group 
 and $X$ a smooth manifold on which $G$ acts smoothly. Such $X$ will be called a
\emph{differentiable $G$-space}.  For any continuous $G$-module $V$, where $V$ is a locally convex topological vector space (as  in Section \ref{sec:cont_cohomology}),  let $$\C^{\infty}(X,V)$$  denote the vector space of smooth functions on $X$ with values in $V$. 
Endowed with the $\Cs$-topology of maps,  
$\C^{\infty}(X,V)$ is a locally convex vector space and (quasi-) complete if $V$ is (quasi-) complete \cite[Chapter 1, \S 10]{Groth_tvs}, see also  \cite[\S 3]{Cass} and \cite{Schwartz}. 
Moreover, since $G$ acts smoothly on $X$, with  respect to \eqref{eq:induced_action},  $\C^{\infty}(X,V) $ becomes a continuous $G$-module in an obvious way. 

\smallskip 
Van Est and Mostow-Hochschild \cite{HM} introduced the notion of \emph{differentiable $G$-modules} and differentiable cohomology groups $ H^{r}_{d} \!   \left(G, \cdot\right)$ based on smooth co\-chains (see Section \ref{sec:diff_cohomology} below for definitions). In particular, if  $X$ is a differentiable $G$-space and $V$  a differentiable $G$-module then $\C^{\infty}(X,V)$ is a differentiable $G$-module. 

\smallskip
The main result for this section will be:

\begin{theorem}[Vanishing theorem, smooth case]  \label{thm:smooth_vanish}
Let $X$ be a differentiable $G$-space on which $G$ acts properly, and let\/ $V$ be a differentiable $G$-module. Then  
$$  H^{r}_{d}  \left(G, \C^{\infty}(X,V) \right)   \, =  \,   H^{r}  \left(G, \C^{\infty}(X,V) \right)  \, =  \zsp \, , \;  r \geq 1. $$
\end{theorem}

This result  implies Theorem \ref{thm:main_cohomology} in the introduction. 

\medskip 
Note that Theorem \ref{thm:smooth_vanish} 
is a differentiable version  of 
Corollary \ref{cor:smooth_man}.  
Here we are dealing with smooth functions as coefficients instead of 
continuous functions.  Likewise, the differentiable cohomology groups $
H^{*}_{d}\left(G, \cdot  \right)$ are using smooth cochains in their definition, and it is required that the coefficient modules for the differentiable cohomology functor $H_d^*(G, \cdot)$  are differentiable $G$-modules. We shall explain these notions right away in the following Section \ref{sec:diff_cohomology}.


\subsection{Differentiable cohomology groups} \label{sec:diff_cohomology}
Since $G$ is a Lie group we can introduce a smooth analogue of the continuous cohomology theory, which was first systematically studied by Van Est  \cite{VanEst}. Its foundations were further developed by Mostow and Hochschild, see \cite{Mostow, HM}. Another good reference is \cite{Blanc}.
The differentiable cohomology of $G$ is a functor defined on the category of
differentiable $G$-modules. 

\smallskip \paragraph{\em Differentiable $G$-modules} 
Let $V$ be a differentiable $G$-module. By this we mean a continuous $G$-module $V$
(with all the assumptions of Section \ref{sec:cont_cohomology} in place, in particular $V$ is a quasi-complete Hausdorff locally convex topological real vector space) 
that satisfies (see \cite{Blanc,HM}) that,  for all $v \in V$, 
\begin{enumerate}
\item the orbit map $G \to V$, $o_{v} \! : g \mapsto g \cdot v$ is smooth.
\item the map $V \to \Cs(G, V)$, $v \mapsto  o_{v}$ is smooth. 
\end{enumerate} 
An isomorphism of differentiable $G$-modules is a $G$-equivariant isomorphism of topological vector spaces. 

\smallskip \paragraph{\em Smooth functions on $G$-spaces} 
A differentiable $G$-space is a smooth manifold $X$ on which $G$ acts smoothly.
In this situation, $$ \Cs(X,V)$$  with the usual action (defined by \eqref{eq:induced_action}) is a differentiable $G$-module. (Compare  \cite[\mbox{$8^{\circ})$} Proposition]{Blanc}.)

\smallskip \paragraph{\em Smooth cochains} 
The differentiable cohomology groups  
of $G$ with coefficients 
in the differentiable module $V$ 
 are defined by using differentiable cochains
instead of continuous cochains.  
For this, we consider in the complex of continuous inhomogeneous 
cochains $$\left(C^{*}(G;V), \partial \right)$$ (see Section \ref{sect:cc_groups})
 the subcomplex of differentiable inhomogeneous cochains 
$$\left( \, C^{*}_{d}(G;V), \partial \, \right) \,  , \; \text{ where } \, C^{r}_{d}(G;V) :=  \, \C^{\infty}(G^{r},V) \,  .$$ 
We put $B^r_{d}(G;V)  = \partial \left(C^{r-1}_{d}(G;V)\right)$ and 
$Z^*_d(G;V) = \ker \partial \cap C^*_d(G;V)$. This defines the differentiable cohomology groups 
$$  H^{r}_{d}(G,V) = \, Z^r_d(G;V) \big/ B^r_d(G,V) \,  . $$   

\smallskip 
We mention the following
comparison result with continuous cohomology. 

\begin{theorem}[Hochschild, Mostow \mbox{\cite[Theorem 5.1]{HM}}] \label{thm:comparison}
 Let $G$ be a real  Lie group and $V$ a differentiable $G$-module. Then the natural 
 map $$    H^{*}_{d}(G,V)  \, \to  \,  H^{*}(G,V)$$
 is an isomorphism of  (topological) vector spaces. 
\end{theorem}

\subsubsection{Compact Lie groups} 
Let $V$ be a differentiable $G$-module. By Proposition \ref{compsta}, 
vanishing of continous cohomology with coefficients  
in $V$ follows for all compact Lie groups $G$. 
In fact, the same proof (or application of Theorem \ref{thm:smooth_vanish}) and Theorem \ref{thm:comparison}) shows: 

\begin{lemma}  \label{lem:smooth_van_compact}
Let $G$ be a compact Lie group and let $V$ be a differentiable $G$-module. 
Then $ H_{d}^{r} \left(G, V\right)   = \zsp$, for all $r \geq 1$.
\end{lemma}

\subsubsection{Smooth Shapiro lemma} \label{sec:smoothShapi}
Let $H$ be a closed subgroup of $G$ and $W$ a differentiable $H$-module. Put 
$$  \Inds_H^G \, W = \; \{f\in \Cs(G,W)\mid f(g h^{-1})=h \cdot f(g), \ g\in G, \, h\in H\}.$$
Then the space $\Inds_H^G \, W$ turns into a differentiable $G$-module, by declaring 
$$  (\alpha \cdot f) \left(g\right) =  f \left(\alpha^{-1} g\right) \; . $$
As it turns out the usual proof of Shapiro's lemma (Proposition \ref{Shapi})  works for differentiable cochains with respect to the functor $ \Inds_H^G$, compare 
\cite{Blanc}.  It relies on a differentiable equivariant version of Frobenius reciprocity related to Lemma \ref{lem:ind2_smooth} below. 

\smallskip 
Let $W$ be a differentiable $H$-module. 
\begin{pro}
[Smooth Shapiro lemma, see \mbox{ \cite[Th\'eor\`eme 11]{Blanc}}]
\label{Shapi_smooth} 
There is a natural isomorphism 
$$  H^r _{d}  \left(G,  \Inds_H^G \, W \right)   \; \cong  \;  H^r_{d}  \left(H, W\right) \, , \, r \geq 0 .  $$
\end{pro}

\smallskip 
\paragraph{\em Remark} Note that, in the view of Theorem \ref{thm:comparison}, Proposition \ref{Shapi_smooth}  implies that the continuous cohomology groups  $H^r \! \left(H,W \right)$ and $H^r \! \left(G,  \Inds_H^G \, W \right)$ are isomorphic.

\smallskip 
\paragraph{\em Associated bundles over $\, G/H$} Let $Y$ be a differentiable $H$-space. Then the associated bundle 
$$ G \times_{H} Y$$ 
is a smooth $G$-manifold, and a locally trivial smooth 
fiber bundle over $G/H$ with fiber $Y$. 
The following is the differentiable analogue of Lemma \ref{lem:ind}.

\begin{lemma} \label{lem:ind_smooth}
Let $V$ be a differentiable $G$-module.
There is a natural  isomorphism of differentiable  $G$-modules 
\begin{equation} \label{eq:inds}
\mu: \, \Inds_{H}^{G} \; \Cs(Y,\res_H^G \, V)  \, \to  \, \Cs(G \times_{H} Y,V) \;   .  
\end{equation}
\end{lemma}   
\begin{proof} First note that for any two smooth manifolds $A, B$, we have (\cf \cite[Ch. 1, \S 10 and
Ch. 3. \S 8]{Groth_tvs}) the adjoint formula
\begin{equation} \label{eq:adjoints} 
 \Cs(A, \Cs(B,V))  \, =  \, \Cs(A \times B,V )  \; . 
\end{equation} 

Also, since $H$ acts properly and freely on $G \times Y$, the differentiable slice theorem 
(compare \eqref{eq:stube}), shows that the quotient map $$G \times Y  \to G \times_{H} Y$$ is a smooth  
$H$-principal bundle map which is locally trivial. 

Now, for any trivial $H$-principal bundle $A \times H$,   
$\Cs(A \times H,V)^{H}  =  \Cs(A,V)$. This allows to show the homeomorphism 
\begin{equation} \label{eq:spushfw} 
\Cs(G \times Y,V)^{H} \to \Cs(G \times_{H} Y,V) \end{equation}
 (compare Lemma \ref{lem:cpushfw}). 
The proof of Lemma \ref{lem:ind_smooth} carries now through analogously to the one of Lemma \ref{lem:ind} (1), using  \eqref{eq:adjoints} and \eqref{eq:spushfw}.  
\end{proof} 

\begin{corollary} [Smooth Shapiro lemma for actions on functions]
\label{cor:Shapi_smooth} 
There  is a natural isomorphism 
$$  H^r_{d}  \left(H,  \Cs(Y,\res_H^G V) \right) \, \to  \, H^{r}_{d} \left(G, \Cs(G \times_{H} Y,V) \right) \; . $$
\end{corollary}

\smallskip 
\paragraph{\em Equivariant reciprocity lemma}
Let $X$ be a differentiable $G$-space. Then $X$ is also a differentiable $H$-space by restriction. We denote this space by $\res_H^G \, X $. 

\smallskip
We mention the analogue of Lemma \ref{lem:ind} (2):  

\begin{lemma} \label{lem:ind2_smooth}
There is a natural isomorphism of differentiable $G$-modules 
$$  \nu:  \Cs\left( X, \Inds^{G}_{H} \, W \right)  \, \to \,  \Inds^{G}_{H} \; \,  \Cs \left( \res_H^G \, X, W \right) \; . $$ 
\end{lemma}

\subsubsection{Local vanishing for smooth coefficients} 
\hspace{1cm} \smallskip

\paragraph{\em Differentiable slices and smooth $G$-tubes} Let  $\cS_{p}$ be a differentiable slice at $p \in X$, 
and $U_{p} =  G \cdot \cS_{p}$  the corresponding smooth $G$-tube. 
By definition a differentiable slice $\cS_{p}$ is a submanifold such that the natural map 
\begin{equation} \label{eq:stube} G \times_{G_{p}} \cS_{p} \to U_{p}
\end{equation} is a  $G$-equivariant diffeomorphism (see \cite[\S2 Lemma 4]{JK}).
As before, for any,  $G$-invariant open subset $U \subseteq U_{p}$ of $X$, $p \in U$, define  $ \cS_{U} =   \cS_{p} \cap U$, which is an open submanifold of the manifold $ \cS_{p}$. Moreover,  
then   $ \cS_{U}$ is a differentiable slice with tube $U$.  In fact, since  \eqref{eq:stube}  is a  $G$-equivariant diffeomorphism,  
so are the restricted maps 
$G \times_{G_{p}} \cS_{U} \to U$. 
In particular, if a tube $U_{p}$, exists we may always find a differentiable slice $S_{U}$ near $p$ which has compact closure.  That is,  by taking $U = G \cdot S$, where $S$ is  a $G_{p}$-invariant neighborhood of $p$ in $S_{p}$ with compact closure in $S_{p}$. 

\begin{lemma}[Smooth $G$-tubes] \label{lem:tubesc_smooth}
There is a natural  isomorphism of differentiable $G$-modules 
\begin{equation} \label{eq:ind_tubes}
{\mathrm{Ind}_{G_{p}}^G \, \C^\infty}(\cS_{U},\res^G_{G_p} V)  \, \to  \, \C^{\infty}(U,V) \,  \text{ and }     
\end{equation}
\begin{equation}\label{eq:smallvanish2}
H^r _{d}\left(G,\C^{\infty}(U, V)\right)= \zsp ,  \; \text{ for all } r \geq 1.
\end{equation}
\end{lemma}
\begin{proof}
By Lemma \ref{lem:ind_smooth}, $ \Inds_{G_{p}}^G \, \C^{\infty}(\cS_{U} \, , \res^G_{G_p} V) = \C^{\infty}\left (G \times_{G_{p}} \cS_{U} \,,\, V \right)$. 
Using the differentiable Shapiro's lemma (Proposition \ref{Shapi_smooth}) we conclude 
\begin{equation*}
 H^r_{d} \left( G, {\Inds}_{G_{p}}^G \, \C^{\infty}(\cS_{U},V) \right) = 
H^r_{d}\left (G_p, \C^{\infty}(\cS_{U}, V) \right) \; . 
\end{equation*}
Since $G_p$ is compact, we have 
$H^r \left(G_p,\Cs(\cS_{U}, V) \right) = \zsp$, $r\geq 1$. 
\end{proof} 

\subsection{Proof of Theorem \ref{thm:smooth_vanish}}
\hspace{1ex} \smallskip

\noindent 
Let $X$ be a differentiable $G$-space and $V$ a differentiable $G$-module. 
We are also assuming that $X$ is a proper $G$-space. As before we let
 $$\pi: X \to X/G$$  denote the quotient map for the $G$-action. 
Then $X/G$ is
a locally compact, paracompact Hausdorff space, since $G$ acts properly.
Moreover, by the differentiable slice theorem \cite[\S2 Lemma 4]{JK} differentiable 
slices do exist for every $p \in X$.  Therefore,  
the local vanishing property for the differentiable cohomology of $G$ 
with coefficients in the function space $ \C^{\infty}(X, V)$ is satisfied
by application of Lemma \ref{lem:tubesc_smooth}. 

\smallskip   
\paragraph{\em Differentiable structure sheaves}
Following Section \ref{sect:sheaves_c}, we are now considering the properties of various sheaves of functions on $X$ and $X/G$ which are associated to the action of $G$.
First let
$$  \Cs_{X,V}  = \Cs( \cdot, V)$$  
denote the sheaf of smooth functions on $X$ with values in 
$V$. Since $G$ acts on $X$ and $V$, the sheaf $\Cs_{X,V}$ has an action of $G$ by co-morphisms which is defined by  \eqref{eq:induced_action}. Furthermore, let  $$ \Cs_{X}  \text{ and } \Cs_{X/G}$$  denote the structure 
sheaves of smooth real valued functions on $X$ and $X/G$,  respectively. 
Note that $X/G$ is in general not a smooth manifold, but we can define 
$$ \Cs_{X/G} =  \pi_{*}^{G} \,  \Cs_{X} \; . $$   
Similarly, for any differentiable $G$-module $V$, 
we have the equivariant direct image sheaf  $\pi_{*}^G \, \Cs_{X,V}$
(compare Section \ref{sect:sheaves_c}).

\begin{lemma} \label{lem:fine_s}
The sheaves $\pi_{*}^G \, \Cs_{X,V}$ are fine sheaves.
\end{lemma} 
\begin{proof}  Since  $\pi_{*}^G \, \Cs_{X,V}$ is a sheaf of $\Cs_{X/G} =  \pi_{*}^{G} \,  \Cs_{X} $-modules, it is sufficient to show that the latter is a fine sheaf of rings (see for example \cite[\S 9]{Bredon_sheaves}). 
This amounts to constructing, for any
covering of the form $\{ \pi^{-1}(U_{j})\}$ of $X$, where 
$\{U_{j}\}$ is a  locally finite covering of  $X/G$,  an associated subordinate partition of unity 
by $G$-invariant smooth functions $\{ \tilde \eta_{j} \in \Cs({X})^{G}\}$. 
To this end,  it is sufficient to construct a partition of unity  subordinate to any locally finite covering of $X$ by $G$-invariant open subsets that refines the covering $\{ \pi^{-1}(U_{j})\}$.  
Therefore, since $X$ is covered by $G$-tubes, 
we may assume that $\pi^{-1}(U_{j}) = G \times_{G_{j}} S_{j}$  is a 
differentiable $G$-tube and 
that $S_{j}$ has compact closure. In particular, $U_{j}$ has compact closure 
in $X/G$. 

Since $X/G$ is paracompact, the covering $\{U_{j}\}$ 
has a subordinate continuous partition of unity $\epsilon_{j}$.  
Therefore the functions
$\epsilon_{j} \circ \pi \in \C(X)^{G}$ give a continuous 
$G$-invariant partition of unity subordinate to $\{\pi^{-1}(U_{j})\}$. 
   
Moreover, we may arrange things that there exist open $G$-invariant 
subsets $W^{1} \subset W^{2} \subset S_{j}$, which in some coordinate system 
are diffeomorphic to Euclidean balls of radius $1$, respectively $2$, such that
$$K_j = \supp{\;  {(\epsilon_j \circ \pi)}_{| \cS_j}}  \subset W^{1} . $$

Using the same technique as employed in the proof of  \cite[Ch. VI, \S4 4.2 Theorem]{Bredon_compact_trans} (for reference on approximation of continuous functions by smooth functions, see 
 \cite{MI} and \cite[Theorem 2.2]{Hirsch}),
we may approximate
$\epsilon'_j$ by a smooth function $$ \eta'_j:S_j\,\ra \,  \RR$$ such that
\begin{enumerate}
\item[(1)]  $\eta'_j$ is a smooth (positive) function on $K_{j}$, 
 \item[(2)] $\eta'_j=0$ on $\cS_j -  \overline{W}^{1}$ (where  $ \overline{W}^{1}$ denotes the closure of $W^{1}$).
\end{enumerate}

Next let $\eta''_j \in \C^{\infty}(\cS_{j}, \RR)$ be defined by
\[
\eta''_j(x)=\int_{G_j}\eta'_j(gx) \,  dg.\]
Then $\eta''_j$ is a $G_j$-invariant nonnegative smooth function.
Since $W^{1}$ is $G_j$-invariant, $\eta''_j$ also satisfies  $(1)$, $(2)$.
In particular, note that 
\[ \supp {\epsilon'_j} \, \subseteq \, \supp {\; \eta''_j} \; . \]
As the restriction map
$\displaystyle \Cs(\pi^{-1}(U_j),\RR)^G\, \to \,  \Cs(S_j,\RR)^{G_j}$
is bijective (as follows from \eqref{eq:ind_tubes}, for example),  
we obtain a
$G$-invariant smooth function $\eta_j$ on $\pi^{-1}(U_j)$, which restricts to $\eta''_j$
on $\cS_j$.

With the above provisions in place, and taking into account that $\pi: \cS_{j} \to U_{j}$ is a quotient 
map (see \cite[Ch. VI, Proposition 3.3]{Bredon_compact_trans}), it follows that $G \cdot (\cS_j -  \overline{W}^{1})$ is an open subset of 
$\pi^{-1}(U_j)$ and a neighborhood of the boundary of the open subset $G \cdot  W_{2}$.  
Therefore, extension by $0$ outside $\pi^{-1}(U_j)$ shows that $\eta_j$ 
arises as the restriction of a smooth  $G$-invariant function $\eta_{j}$ defined on $X$ 
with support in $\pi^{-1}(U_j)$. 

By construction, $ a = \sum_{j} \eta_j  \in \Cs(X,\RR)$, 
 is everywhere positive on $X$
and $G$-invariant. Thus $$ \tilde \eta_j =\frac 1{a} \,  \eta_j$$  defines
a partition of unity by $G$-invariant smooth functions subordinate 
to the covering by $G$-neighborhoods $\{ \pi^{-1}(U_j) \}$.
\end{proof} 

We conclude the proof of Theorem \ref{thm:smooth_vanish} using the reasoning developed in Section \ref{sect:sheaves_c} as follows:

\smallskip
First of all, in the view of \eqref{eq:smallvanish2} we have
a resolution of sheaves
\begin{equation} \label{eq:resolution_s}
\zsp  \, \lra  \,   \pi_{*}^G \, \Cs_{X,V}   \lra  \, (C^\infty_{X,V})^0  \stackrel{\partial^{0}}{\lra} \, (C^\infty_{X,V})^1   \stackrel{\partial^{1}}{\lra}\,  (C^{\infty}_{X,V})^2 \stackrel{\partial^{2}}{\lra}  \ldots  \, \; \; , 
\end{equation} 
where $(C^\infty_{X,V})^r$ are sheaves on $X/G$ whose sections over $U$ are differentiable $r$-cochains of $G$ with coefficients in smooth functions $\Cs(\pi^{-1}(U),V)$.  Since the sheaves $(C^\infty_{X,V})^r$ are sheaves of $\Cs_{X/G}$ modules (which is a fine sheaf of rings by Lemma \ref{lem:fine_s}) these are fine sheaves. It follows 
\begin{equation} \label{eq:coho_ident_s}
 H^{r}_d \left(G, \Cs(X,V) \right) \, \cong \, H^{r} \left( \pi_{*}^G \, \Cs_{X,V} \right) \, . 
\end{equation}
Finally, by Lemma \ref{lem:fine_s}, $\pi_{*}^G \, \Cs_{X,V}$ is a fine sheaf. It follows that the right hand cohomology in \eqref{eq:coho_ident_s} is acyclic. Therefore,
$$ H^{r}_d \left(G, \Cs(X,V) \right) = \zsp\,  , \, r \geq 1 \, . $$
This finishes the proof of Theorem \ref{thm:smooth_vanish}.

\section{Proof of Theorems \ref{thm:main_CR} and Theorem \ref{thm:main_conformal}} \label{sec:proofs_main}

Given a pseudo-Hermitian structure $\{ \om, J \}$ on $M$, the distribution $$ D = \ker \om$$ defines  a strictly pseudo-convex $CR$-structure $\{D, J \}$ on $M$. We have the following naturally  associated transformation groups, namely
the group  $$ \Psh(M,\{\om,J\}) =\{\, \alpha\in\Diff(M)\mid \alpha^{*}\om=\om, \;  \alpha_*\circ J=J\circ \alpha_*|_{\sfD} \,\} $$ of pseudo-Hermitian transformations and the group  of $CR$- automorphisms 
$$ \Aut_{CR}(M, \{\sfD,J\}) =\{ \, \alpha\in\Diff(M)\mid \alpha_*\sfD=\sfD, \;   \alpha_*\circ J=J\circ \alpha_*|_{\sfD} \, \} \; \;  . $$
In general, the inclusion $$\Psh(M,\{\om,J\}) \, \leq \, \Aut_{CR}(M, \{\sfD,J\})$$ is strict, since the contact form $\om$ is determined
by $D$ only up to conformal equivalence. 
Moreover, the Lie group $\Psh(M,\{\om,J\})$ always acts properly on $M$, whereas in some  cases 
(as detailed in Theorem \ref{thm:schoen}) the $CR$-automorphism group $\Aut_{CR}(M, \{\sfD,J\})$ is too large and doesn't act properly on $M$. 
About the possible relations of the group $\Psh(M,\{\om,J\})$ and  the $CR$-automorphism group 
$\Aut_{CR}(M, \{\sfD,J\})$, we shall prove: 

\begin{theorem} \label{thm:proper_cr}
Let $(M,\{\sfD,J\})$ be a strictly pseudo-convex $CR$-manifold and let 
$  G \, \leq \, \Aut_{CR}(M, \{D,J\})$ 
be a subgroup of $CR$-automorphisms that acts properly on $M$. Then there exists on $M$ a  pseudo-Hermitian structure $\{ \eta, J\}$ compatible with $\{\sfD,J\}$, such that $ G \, \leq \, \Psh(M, \{ \eta, J\} ) \, . $ \end{theorem}

Together with Schoen's theorem (Theorem \ref{thm:schoen}), this result obviously implies Theorem \ref{thm:main_CR} in the introduction.  

\medskip 
This section is organized as follows: In subsection \ref{sec:crossed_hom}  we prepare the proof of Theorem \ref{thm:proper_cr} with a brief discussion of equivariant crossed homomorphisms which are associated to group actions on manifolds. Following that we construct the canonical cohomology class associated to the action of $CR$-automorphisms. The proof of Theorem  \ref{thm:proper_cr} is based on the vanishing of this class, see section  \ref{sec:cr}.  Analogous results for the conformal case and also for locally conformal K\"ahler manifolds will be discussed in subsections  \ref{sec:conf} and  \ref{sec:lcK}.

\smallskip 
\subsection{Equivariant crossed homomorphisms} \label{sec:crossed_hom}
Let $G$ be a Lie group which acts smoothly on the manifold $X$. We let 
$\RR^{>0}$ denote the multiplicative group of positive real numbers.
Next consider $$\Cs(X,\RR^{>0})$$ the group of all smooth maps from $X$ into $\RR^{>0}$ endowed with its natural $G$-module structure, where, for $\al\in G$, $f\in \Cs(X,\RR^{>0})$, 
\begin{equation}\label{eq:mapusua}
(\al \cdot f)(x)= f(\al^{-1}x) \, . 
\end{equation}
In fact, taking the $\Cs$-topology of maps, 
$\Cs(X,\RR^{>0})$ is a differentiable $G$-module in the sense of Section \ref{sec:diff_cohomology}, and 
it is isomorphic to the $G$-module $\C^\infty(X,\RR)$.
Concerning the associated \emph{differentiable cohomology} groups of $G$, 
we note:


\begin{theorem}  \label{thm:vanish}
Suppose that $G$ acts properly on $X$. Then, for all $r \geq 1$, we have 
$ \displaystyle  H^r_d \left(G , \Cs(X,\RR^{>0})\right) = \zsp$. 
\end{theorem}
\begin{proof}
The map $\exp: \Cs(X,\RR) \, \ra \, \Cs(X,\RR^{>0})$
defined by $\displaystyle u=\exp f$,
that is,
\[ u(x)=e^{f(x)}\, , \;  x \in X \]
is easily seen to be an isomorphism of differentiable $G$-modules. In fact, the correspondence is $G$-equivariant, since, for all $\alpha \in G$, 
\begin{equation*}
\exp{(\al\cdot f)(x)} =e^{f(\al^{-1}x)} =u(\al^{-1}x)=(\al\cdot u)(x) \, .
\end{equation*}
Hence, for all $r$, there is an isomorphism of differentiable cohomology groups 
\begin{equation*}
H^r_d\left(G, \Cs(X,\RR)\right)\cong 
H^r_d\left(G, \Cs(X,\RR^{>0})\right)
\, .
\end{equation*}
Since $G$ acts properly, Theorem \ref{thm:vanish} implies
\begin{equation*}
H^r_d\left(G, \Cs(X,\RR^{>0})\right)= \zsp 
\, , \ r\geq 1.  \qedhere
\end{equation*}
\end{proof}

\subsubsection{Differentiable crossed homomorphisms}
Smooth one-cocycles $$\lambda \in \Cs\left(G,  \Cs(X,\RR^{>0})\right) \, , \; \partial^1 \lambda = 0$$ are representing the elements of the first \emph{differentiable}  cohomology group $$ H^1_{d}\left(G, \Cs(X,\RR^{>0})\right) \, .  $$
The condition $\partial^1 \lambda = 0$ amounts to 
the requirement that, for all $\al,\be\in G$, 
\begin{equation} \label{eq:crossed_hom}
\lam(\al  \be) \, (x)= \lam (\be)\,(\alpha^{-1} x )  \cdot \lam(\al) \, (x) \ \, . 
\end{equation} 
Denoting by $\alpha_*$ the covariant map on forms which is induced by $\alpha$, this relation can be written in the concise form 
\begin{equation} \label{eq:crossed_homc}
\lambda(\alpha \beta) = \, \alpha_* \, \lambda(\beta)\,\cdot   \lambda(\alpha) . 
\end{equation}
Such $\lambda$ are called \emph{differentiable crossed 
homomorphisms}.

\smallskip
\paragraph{\em Exact crossed homomorphisms}
For any crossed homomorphism $\lambda$, the cohomology class $$[\, \lambda \, ] \, \in \, H^1_{d}\left(G, \Cs(X,\RR^{>0})\right)$$ vanishes 
if there exists $ v\in \Cs(X,\RR^{>0})$, with
$ \partial^0 v=  \lam$. That is, if 
\begin{equation} \label{eq:coboundary} 
(\al_{*}  v)  \cdot  v^{-1}  =\lam(\al) \, , \text{ for all } \al \in G , 
\end{equation}
or,  equivalently, $\lambda(\alpha)(x)  = v(\alpha^{-1} x) \cdot v^{-1}(x)$, for all
$x \in X$.  
Moreover, two 
crossed homomorphisms $\lambda$ and $\lambda'$ represent the 
same class in  the group $H^1_{d}\left(G, \Cs(X,\RR^{>0})\right)$ if and only if 
$ \lambda' = \lambda \cdot   \partial^0 v$, for some $v \in  \Cs(X,\RR^{>0})$. 


\subsection{$CR$-automorphisms and canonical class} 
\label{sec:cr}
By definition, for any pseudo-Hermitian manifold  $(M,\{\om,J\})$,
the Levi form of the underlying $CR$-structure $$  B= d\om \left(J \cdot ,\cdot \right)$$  
is \emph{positive definite} on $D = \ker \om$.
For any $CR$-automorphism  $$ \alpha \in\Aut_{CR}(M,\{\sfD,J\}) \, , $$ the equation $\alpha_*({\sfD})={\sfD}$
implies that there exists $ u_\alpha \in \C^\infty(M, \RR^{>0})$ with  
\begin{equation} \label{eq:posi}
\alpha^*\om= u_\alpha\cdot\om \; . 
\end{equation}
(In fact, to show $u_\alpha>0$, note that, for any $0 \neq v \in D= \ker \om$, we have $$ B(\alpha_* v, \alpha_* v) = d\om (J \alpha_* v, \alpha_* v) = d \, \alpha^* \! \om \left(Jv, v \right) = u_\alpha B(v, v) \, > 0 .\;  ) $$ 

%

\subsubsection{Associated crossed homomorphism and canonical cohomology class}
Let $$ G  \, \leq \, \Aut_{CR}(M,\{\sfD,J\})$$  be a group of $CR$-automorphisms. Next choose a compatible pseudo-Hermitian structure   $\{\om,J\}$ for the $CR$-structure on $M$. Define a map
$$ \lam_{\mathrm{CR}}: G  \, \to \, \Cs(M,\RR^{>0})$$
by declaring   
\begin{equation}\label{crossedhomo_cr}
  \lam_{\mathrm{CR}}(\al)\, = \, \alpha_* u_{\al}\,  , \; \,  \al\in \Aut_{CR}(M) \, ,
\end{equation}
where $u_{\alpha}$ is as in \eqref{eq:posi} defined relative to $\om$. (More explicitly, \eqref{crossedhomo_cr}  amounts to 
$ \lam_{\mathrm{CR}}(\al)\, (x)= u_{\al}\, (\al^{-1}x) $, for all $x \in M$.)

\smallskip
We claim that  $\displaystyle  \lam_{\mathrm{CR}}$ is a crossed homomorphism for $G$. 

\begin{definition}
The map $ \lam_{\mathrm{CR}}$ is  called {crossed homomorphism  associated to $G$} and the pseudo-Hermitian structure $\{\om,J\}$. Its cohomology class $\mu_{\mathrm{CR}}$ is defined by the underlying $CR$-structure only and it is called the  $\emph{canonical class}$ associated to the $CR$-action of $G$.  
\end{definition} 

\medskip 
The geometric 
meaning of the canonical cohomology class is given by: 

\begin{pro}[Cohomology class of $CR$-transformation group]  \label{pro:can_class_CR}
Suppose that $M$ is a stric\-tly pseudo-convex $CR$-manifold and let $G$ be a Lie subgroup of the group of $CR$-automorphisms of $M$. Then: 
\begin{enumerate}
\item There exists in the differentiable cohomology of $G$ a natural associated class 
$$ \mu_{\mathrm{CR}} =  [\lambda_{\mathrm{CR}}]  \; \in \, H^1_d\left(G, \C^\infty(M,\RR^{>0})\right)$$ 
which is induced by the $CR$-structure on $M$. 
\item 
Moreover, the class $\mu_{\mathrm{CR}}$ vanishes if and only if there exists a contact form $\eta$ compatible with the $CR$-structure,  such that $G$ is contained in the group of pseudo-Hermitian transformations 
$\Psh(M,\{\eta,J\})$.	
\end{enumerate} 
\end{pro}
\begin{proof} 
To show that $\lambda_{CR}$ is a crossed homomorphism,  we calculate  
\begin{equation*}\begin{split}
(\al \cdot  \be)^*\, \om&= \be^* (\al^*\om)=\be^*\, (u_\al\, \om)= (\be^*u_\al) \cdot  u_\be \;   \om.
\end{split}
\end{equation*}
 As, according to \eqref{eq:posi}, $(\al \cdot \be)^*\om= u_{\al\be}\,\om$,  for some $u_{\al\be}\in \C(X,\RR^{>0})$,
we have $
u_{\al\be}\, = \be^* u_\alpha \cdot u_\be$. 
In particular, 
\begin{equation*}\begin{split}
 \lam_{\mathrm{CR}}(\al\be) &= (\al \cdot \be)_* \, u_{\al\be} =
\al_* u_\al \cdot (\al \cdot \be)_* \, u_\be 
=  \lam_{\mathrm{CR}}(\al) \cdot \al_*  \lam_{\mathrm{CR}}(\be) \, . 
\end{split}
\end{equation*}
So $\lam_{\mathrm{CR}}$ is a crossed homomorphism as in \eqref{eq:crossed_homc}. 

Suppose next that $\omega = v \, \omega'$, for some $v \in  \C^\infty(M,\RR^{>0})$. From 
\begin{equation*} 
u_{\alpha} \,  \om  =  \al^{*} \om  =  \al^{*} (v  \, \omega') =    \al^{*} v \cdot u'_{\alpha} \,  \om' =   \al^{*} v \cdot u'_{\alpha} \cdot  v^{-1}  \, \om
\end{equation*}
we deduce $u_{\alpha} =  \al^{*} v \cdot u'_{\alpha} \cdot  v^{-1}$. Thus $\lambda_{\mathrm{CR}}(\al) = 
\alpha_{*} u_{\alpha} = v \cdot  \alpha_{*} v^{-1} \cdot \alpha_{*} u'_{\alpha} = (\partial^{0} v^{-1}) (\al) \cdot  \lambda'_{\mathrm{CR}}(\al) $, showing that the associated crossed homomorphisms $\lambda_{\mathrm{CR}}$ and $\lambda'_{\mathrm{CR}}$ are representing the same cohomology class $\mu_{\mathrm{CR}}$. 
\smallskip 

%

\smallskip
Next let $\{\om,J\}$ be a pseudo-Hermitian structure on $M$ compatible with the $CR$-structure and $\lambda_{\mathrm{CR}}$ be the associated crossed homomorphism defined by $\om$. Suppose that $[ \lambda_{\mathrm{CR}} ] = 0$. Therefore, as in \eqref{eq:coboundary}, there exists $v\in \Cs(M,\RR^{>0})$
such that $ \displaystyle \alpha_* v =\lam(\al) \cdot v \, , \text{ for all } \alpha \in G$, which amounts to 
\begin{equation}\label{eq:cobound_cr}
\begin{split}
v &= \alpha^* \, (\lam(\al) \cdot v ) = u_\al \cdot  \al^*v \, .
\end{split}
\end{equation}
Put a $1$-form $
\eta=v\cdot \om $. 
Then $$ \al^*\eta=\al^*v\cdot \al^*\om=\al^*v \cdot u_\al \, \om=v\cdot \om=\eta \, .$$ 
Hence, $ \alpha \, \in  \,  \Psh(M,\{\eta,J\})$ is equivalent to \eqref{eq:cobound_cr}. This  shows (2). 
\end{proof}

\subsubsection{Proof of Theorem \ref{thm:proper_cr}}
Since $G$ acts properly, Theorem \ref{thm:vanish} shows that the canonical crossed homomorphism $\lambda_{\mathrm{CR}}$ for the group of $CR$-\-auto\-mor\-phisms $G$,  as defined in \eqref{crossedhomo_cr},  is exact.
In the view of (2) of Proposition  \ref{pro:can_class_CR}, this  proves Theorem \ref{thm:proper_cr}. 
\hfill \qed

\subsection{Conformal case} \label{sec:conf}
Replacing the role of  $\om$ in $CR$-geometry by a Riemannian metric $g$ on $M$, the  conformal
class of $g$ is said to establish  a \emph{conformal structure} on $M$. Every 
diffeomorphism $\alpha : M\ra M$ that satisfies $\alpha^*g= u_\alpha \cdot g$ for some positive smooth function $u_\alpha \in \C^\infty(M,\RR^{>0})$ is correspondingly called a \emph{conformal automorphism}
of $(M,g)$.  

\smallskip 
As in the $CR$ case, Proposition \ref{pro:can_class_CR}, to any Lie group $G$ of conformal automorphisms 
there is a natural associated cohomology class  $ \mu_{\mathrm{Conf}}$, which is an obstruction for $G$ being a group of isometries: 

\begin{pro}[Cohomology class of conformal transformation group]  \label{pro:can_class_Conf}
Let $(M,g)$ be  a Riemannian manifold and let $G$ be a Lie subgroup of the group of 
conformal automorphisms of $M$. Then: 
\begin{enumerate}
\item There exists in the differentiable cohomology of $G$ a natural associated class 
$$ \mu_{\mathrm{Conf}} =  [\lambda_{\mathrm{Conf}}]  \; \in \, H^1_d\left(G, \C^\infty(M,\RR^{>0})\right)$$ 
which is induced by the conformal structure on $M$. 
\item 
Moreover, the class $\mu_{\mathrm{Conf}}$ vanishes if and only if there exists a Riemannian metric  $h$ conformal to $g$, such that $ G$ is contained in the group of isometries $ \Iso(M, h) $.
\end{enumerate} 
\end{pro}
\begin{proof} We define the cocycle 
$ \lam_{\mathrm{conf}}: G  \, \to \, \Cs(M,\RR^{>0})$
by declaring   
\begin{equation}\label{crossedhomo_conf}
  \lam_{\mathrm{conf}}(\al)\, = \, \alpha_* u_{\al}\,  , \; \,  \al\in \Aut_{CR}(M) \, ,
\end{equation}
where $u_\alpha \in \C^\infty(M,\RR^{>0})$  is defined as above by the relation 
$\alpha^*g= u_\alpha \cdot g$. 
As in the proof of  Proposition \ref{pro:can_class_CR} it can be verified that $\displaystyle  \lam_{\mathrm{conf}}$ is a crossed homomorphism for $G$, and its class $ \mu_{\mathrm{Conf}}$ in  $H^1_d\left(G, \C^\infty(M,\RR^{>0})\right)$ depends only on  the conformal class of $g$. This shows (1).

The  condition  that $[\lam_{\mathrm{Conf}}]  \in H^1_d \left(G , \Cs(M,\RR^{>0})\right)$ vanishes means that there  exists  $v\in \Cs(M,\RR^{>0})$ with $\partial^{0} v = \lambda_{\mathrm{Conf}}$, that is, by \eqref{eq:coboundary}, \begin{equation} \label{eq:cobound_conf}
(\al_{*}  v)  \cdot  v^{-1}  =\lam_{\mathrm{Conf}}(\al) = \alpha_{*} u_{\alpha} \; \text{, for all $\alpha \in G$} 
.
\end{equation}
On the other hand, putting $h = v \cdot g$, we have 
$\alpha^{*} h = \alpha^{*} v \cdot  \alpha^{*} g =   \alpha^{*} v \cdot u_{\alpha} g$. 
Therefore $\alpha^{*} h = h$ is equivalent to \eqref{eq:cobound_conf}, 
which is equivalent $(\partial^{0} v) (\alpha) = \lambda_{\mathrm{Conf}}
(\alpha)$. This implies (2). 
\end{proof}

We also obtain the following analogue of Theorem \ref{thm:proper_cr}: 

\begin{theorem} \label{thm:proper_conformal}
Let $(M,g)$ be a Riemannian manifold. 
Suppose that  
$ G \, \leq  \, \Conf(M,g)$
is a subgroup of conformal automorphisms that acts properly on $M$. Then there exists on $M$ a Riemannian metric $h$ conformal to $g$, such that $ G \, \leq \, \Iso(M, h) \, . $
\end{theorem}
\begin{proof} Since $G$ acts properly,  Theorem \ref{thm:vanish} implies that the associated differentiable cohomology class $ \mu_{\mathrm{Conf}}$ vanishes. Hence,  there exists on $M$ a Riemannian metric $h = v \cdot g $, such that $ G \, \leq \, \Iso(M, h) \, . $
\end{proof}

\subsection{Cohomological characterization of proper actions}  \label{sec:proper_and_coho}

The following theorem shows that the properness of $CR$- and conformal actions is basically a vanishing 
property for differentiable cohomology.  This  also 
implies Corollary \ref{cor:main_cr_and_h1} in the introduction:   

\begin{theorem} \label{thm:proper_and_coho}
Let $G$ be a Lie group of diffeomorphisms of the smooth manifold $M$ that preserves either a strictly pseudo-convex $CR$-structure or a conformal Riemannian structure on $M$.  If $G$ is closed in the group of all diffeomorphisms of $M$, then the following are equivalent: 
\begin{enumerate}
\item $G$ acts properly on $M$. 
\item $H^1_{d}\left(G, \C^{\infty}(M, \RR)\right) = \zsp$.
\item $H^{r}_{d}\left(G, \C^{\infty}(M, V)\right) = \zsp$, for all $r>0$, and any differentiable $G$-module $V$.
\end{enumerate}
 \end{theorem} 
 
\begin{proof} Suppose that $G$ acts properly, then by Theorem \ref{thm:smooth_vanish}, (3) is satisfied. 
Now (3) clearly implies (2).  Finally,  if (2) is satisfied, the canonical class $\mu \in H^{1}_{d}\left(G , \Cs(M,\RR^{>0})\right)$ associated to either the $CR$- or conformal structure on $M$ vanishes. In the first case, as is implied by  Proposition \ref{pro:can_class_CR}, $G$ preserves an associated contact Riemannian metric on $M$,  respectively in the case of Proposition \ref{pro:can_class_Conf}, the group $G$ preserves a Riemannian metric in the given conformal class. Since the group of isometries of a Riemannian manifold acts properly by the theorem of Myers and Steenrod \cite{MS}, and $G$ is a closed group of isometries,  $G$ acts properly on $M$.  
\end{proof}

\subsection{Locally conformal K\"ahler metrics} \label{sec:lcK}
%
%
In this subsection we let  $(X,J)$ denote a connected complex manifold satisfying 
${\rm dim}_\RR\, X=2n\geq 4$. 
Then a Hermitian metric $h$  for  $X$ is called a  \emph{locally conformal K\"ahler} metric if it is \emph{locally} conformal to a K\"ahler metric. 
(This means that 
 there exists an open covering $\{  U_\ell \}$  of $X$ and functions $u_\ell:  U_\ell \ra\RR^{>0}$
such that $h=u_\ell \cdot g_{\ell}$, where  $g_{\ell}$ is a K\"ahler metric on $(U_\ell, J)$, 
compare  \cite{Va}, \cite{DO}.)  
%

\subsubsection{Conformal automorphisms of K\"ahler metrics}
Our first result concerns $lcK$-metrics which admit a K\"ahler metric in 
their conformal class: 

%

%
%
%
%

\begin{theorem}\label{thm:hollcK}
Let $(X,\{J,g\})$ be a K\"ahler manifold which is not holomorphically  isometric to $\CC^n$.
Then the following hold:  \smallskip

{\bf (1)}\,  There exists an $lcK$ manifold  $(X, \{h,J\})$ such that  the $lcK$ metric $h$  
is conformal to the K\"ahler metric $g$ and satisfying 
$$  \Iso\left(X,\{h,J \}\right) \, =  \, \mathrm{Conf}(X,\{g,J\})  
\, . 
$$ 


{\bf (2)}\, Furthermore, the  holomorphic isometry group  $\Iso\left(X,\{h,J \}\right)$  
is maximal among all isometry groups of Hermitian metrics 
conformal to $g$.
\end{theorem}
\begin{proof} 
 
 Since $X$ is not holomorphically isometric to $\CC^{n}$, we infer from Theorem \ref{thm:schoenhol} that the Lie group $ \mathrm{Conf}(X,\{g,J\})$ is acting properly on $X$.  
Therefore, according to Theorem \ref{thm:proper_conformal}, 
there exists a
Riemannian metric $h$ conformal to $g$, such that $  \mathrm{Conf}(X,\{g,J\})  \leq  \Iso(X, h)$.  
Since $h$ is conformal to the K\"ahler metric $g$, $h$ is Hermitian for the complex structure $J$  
and   $ \mathrm{Conf}(X,\{g,J\}) \leq 
 \Iso\left(X,\{h,J \}\right)$.  From the fact that $\mathrm{Conf}(X,\{g,J\}) = \mathrm{Conf}(X,\{h,J\})$,  
 we deduce that in fact  $\mathrm{Conf}(X,\{g,J\}) = 
 \Iso\left(X,\{h,J \}\right)$. This proves (1). 
 
For the proof of (2), note that 
$\Isom \left(X,\{ h',J \}\right) \leq  \, \mathrm{Conf}(X,\{h',J\})  =  \mathrm{Conf}(X,\{g,J\}) =  \Isom\left(X,\{ h,J \}\right)$, by (1). 
This shows (2). 
\end{proof}

\smallskip 
\paragraph{\em Remark} 
For a  K\"ahler manifold  $(X, \{g,J\})$ with
${\rm dim}_\RR\, X=2n\geq 4$,
the group of holomorphic conformal 
diffeomorphisms $$ {\Conf}(X,\{g,J\})$$  coincides with 
the group of holomorphic homothetic transformations 
\begin{equation*} \begin{aligned} 
{\rm Hoth} &  (X,\{g,J\}) = \\  &  \{ \alpha \in \Diff(X)\mid \alpha^* g= c\cdot g,\, \alpha_*J=J \alpha_* \text{,  for some }  c \in \RR^{>0} \} \; . 
\end{aligned}
\end{equation*} 
In fact, any biholomorphic conformal map between K\"ahler manifolds is easily seen to be an isometry up to constant scaling of the metrics, compare  \cite[Theorem 6.5, p. 66]{YANO}. Therefore, (1) of Theorem \ref{thm:hollcK} also asserts the equality 
\begin{equation} \tag{{\bf 1'}} 
  \Iso\left(X,\{h,J \}\right) \, =  \, \mathrm{Hoth} (X,\{h,J\})  
 \, = \,   \Hoth\left(X,\{g,J\}\right) \; . 
 \end{equation}

\subsubsection{Conformal automorphisms of $lcK$-metrics}

We are now looking at the conformal class of $lcK$-metrics and their conformal automorphisms in general. 
First we recall that $lcK$-manifolds admit K\"ahler coverings (compare \cite[p. 65]{Va}): 

\begin{lemma}[K\"ahler covering] Let $X$ be a simply connected complex manifold  and $h$ any $lcK$-metric on $X$. 
Then 
\begin{enumerate}
\item The  $lcK$-metric $h$ is conformal  to a K\"ahler metric $g$  (which is unique up to a constant factor).
\item The  $lcK$-metric $h$ is locally holomorphically conformal to $\CC^{n}$ if and only if 
$g$ is a flat K\"ahler metric (that is, $g$ is locally holomorphically isometric to $\CC^{n}$). 
\item  The  $lcK$-manifold $(X,\{h,J\})$ is holomorphically conformal to $\CC^{n}$ if and only if the K\"ahler metric $(X,\{g,J\})$ is holomorphically isometric to $\CC^{n}$. 
\end{enumerate}
\end{lemma}
\begin{proof} Let $\Theta = h(J \cdot, \cdot)$ be the fundamental two-form of the $lcK$-metric $h$.
There exists a closed (global) one-form $\theta$ on $X$, called Lee form,  such that
$$ \displaystyle d\Theta=\theta\we \Theta \,  . $$
Indeed,  by the definition of $lcK$ metric, $\theta$ is constructed as  $\theta = d \log u_{\ell}$ on $U_{\ell}$, where $h=u_\ell \cdot g_{\ell}$ and  $g_{\ell}$ is a K\"ahler metric on $U_\ell$, for
some covering of $X$. 
Furthermore, assuming that  $X$ is simply connected,
$\theta=df$ for some function $f$ on $X$. 
Then $ \Omega=e^{-f}\cdot \Theta$ is a K\"ahler form on $X$. This proves (1). 

The $lcK$-metric $h$ being locally holomorphically conformal to $\CC^{n}$ means that there exists 
locally a holomorphic conformal map to $\CC^{n}$. 
Since the K\"ahler metric $g$ is conformal to $h$, this map is also locally conformal for $g$, so $g$ is
locally holomorphically conformal to the standard flat K\"ahler space $\CC^{n}$. By the above remark following Theorem \ref{thm:hollcK}, $g$ is actually locally homothetic to the standard complex space $\CC^{n}$, which also implies that $g$ flat and locally holomorphically isometric to $\CC^{n}$,  proving (2). The remark also implies (3). 
\end{proof}

The following is now a consequence of Theorem \ref{thm:hollcK}:

\begin{corollary}\label{cor:hollcK}
Let $(X, \{h,J \})$ be an $lcK$-manifold whose universal  covering manifold  is not holomorphically  conformal to $\CC^n$. Then there exists an $lcK$-metric $h'$ conformal to $h$ which is satisfying 
$$  \Iso\left(X,\{h',J \}\right) \, =  \, \mathrm{Conf}(X,\{h,J\})  \; . 
$$ 
%
%
%
\end{corollary}
\begin{proof} Let $\mathsf{p}: (\tilde X, \tilde h) \to (X,h)$ be the universal covering $lcK$-manifold. Let $g$ be the K\"ahler metric on $\tilde X$ conformal to $\tilde h$. Then the group of decktransformations $\Gamma$ for the covering $\mathsf{p}$ is contained in $\mathrm{Hoth} (\tilde X,\{g,J\})$. Since $(X,g)$ is not  holomorphically isometric to $\CC^{n}$, Theorem \ref{thm:hollcK} implies that there exists a Hermitian metric $\tilde h'$ conformal to $g$ such  that $\Iso(\tilde X,\{ \tilde  h',J  \}) \, =  \, \mathrm{Conf}(\tilde X,\{g,J\})  = \mathrm{Conf}(\tilde X,\{\tilde h,J\})$.  Since $\Gamma \leq \Iso(\tilde X,\{ \tilde  h',J \})$, there exists a unique $lcK$-metric $h'$ on $X$, such that $\mathsf{p}: (\tilde X, \tilde h') \to (X,h')$ is a holomorphic Riemannian covering, and this  metric is conformal to $h$. Moreover, $\Iso(X,\{  h',J  \}) = \mathrm{Conf}(X,\{h,J\})$. 
\end{proof} 

%
 
%

\appendix 

\section{Other related results}

\subsection{Conformally flat K\"ahler manifolds} \label{A1}
 
The following results were  prov\-ed by K.\ Yano and I.\ Mogi  \cite[Theorem 4.1]{MK} in the case  
$\dim_{\RR} X \geq 6$, and by S. Tanno \cite{Tano}  for $\dim_{\RR} X =4$. 
%

\begin{theorem}\label{YM}
Any  conformally flat K\"ahler manifold $X$ of dimension $n\geq 6$ is locally flat, that is,  
it is \emph{locally holomorphically isometric} to $\CC^n$.
\end{theorem}
Here the complex space $\CC^n$ carries the standard flat K\"ahler metric.

\begin{theorem}\label{Tanno}
Any\/ $4$-dimensional  conformally flat K\"ahler manifold $X$ 
is \emph{locally holomorphically isometric} to either $\CC^2$
or the product of $2$-dimensional surfaces $\HH^2_\RR\times S^2$
with constant opposite sign.
\end{theorem}

Recall that the universal covering of any conformally flat  $m$-dimensional manifold admits a 
conformal development map into the sphere $S^{m}$ and this map is unique up to composition 
with an element of $\Conf(S^{m}) =\PO(m+1,1)$, see for example \cite[Theorem 4]{Kui}. 
In the case of the two theorems, the universal covering 
space $\tilde X$ is thus mapped to $S^4-\{\infty\}= \CC^2$ or
the sphere complement $S^4-S^1=\HH^2_\RR\times S^2$  (compare \cite{KA2}) 
through this developing map.

\subsection{Proof of Theorem  \ref{thm:schoenhol}} \label{proofA2}
%
%
We assume  that $  \Conf(X, \{g,J\})  \leq  \Conf(X,g)$ does not act properly on $X$, where 
$g$ is the K\"ahler metric. 

\smallskip 
Therefore, Schoen's theorem \cite{SC} implies that there exists a conformal diffeomorphism from  $X$ 
to either the sphere $S^{2n}$ 
or $\RR^{2n}$. 
Since $X$ is assumed to be K\"ahler, $H^2(X,\RR) \neq \zsp$,
in case $X$ is compact. Therefore, $S^{2n}$, $n>1$, does not occur. 
Thus, in the case $n>2$, there exists a conformal diffeomorphism $\psi: X \to \RR^{2n}$.
In particular, $(X,g)$ is conformally flat and $X$ is simply connected.

%
%
%

\begin{proof}[Proof of Theorem  \ref{thm:schoenhol}]    
\hspace{1cm} \medskip

(i)\,  When $\dim_{\RR}  X\geq 6$, 
by the result of Yano and Mogi Theorem \ref{YM},
the K\"ahler manifold $X$ has everywhere {holomorphic sectional curvature $0$}. In particular, 
it is locally holomorphically isometric to the standard flat complex space
$\CC^{n}$ (compare \cite[IX, Theorem 7.9]{KN2}). 
Since $X$ is simply connected, the usual monodromy argument (see \cite{Kul}) 
shows that there is a holomorphic map  $\varphi: X\,\ra\, \CC^n$, that is also
an isometric immersion.  

Up to a conformal map we may identify both domains $\RR^{2n}$ and $\CC^n$ with an open subset 
$U = S^{2n} - \{ p \}$ contained in $S^{2n}$. Thus 
$\varphi, \psi$ 
correspond to maps 
$$\bar \varphi, \bar \psi: \; X \to U \subset S^{2n}$$ and both $\bar \varphi$, $\bar \psi$ 
are developing maps for  the 
locally flat conformal structure associated with $(X,g)$.
The uniformization theorem 
for locally flat conformal structures  \cite{Kui} 
implies that the two developing maps $\bar \varphi$, $\bar \psi$ 
are \emph{equivalent} 
by an element $\alpha$ of $\Conf(S^{2n})=\PO(2n+1,1)$,
so that $\alpha \circ \bar \psi= \bar \varphi$ on $X$. Clearly, since $\bar \psi$ is a diffeomorphism, this  
shows that $\bar \varphi$ is an injective embedding of $X$ into $U$.
Note that  every conformal 
embedding of Euclidean space $\RR^{2n}$ into $S^{2n}$ has
as image $S^{2n}$ with a point removed. Since the image of 
$\alpha \circ \bar \psi$ is contained in $U$, we conclude that 
$\alpha \circ \bar \psi$ is, in fact, surjective onto $U$.
Thus the holomorphic isometric immersion
$\varphi: X\ra\, \CC^n$ turns out to be an isometry. 
\medskip

(ii)\, $\dim_{\RR} X=4$. By Tanno's result, Theorem \ref{Tanno}, since $X$ is conformally flat
K\"ahler,
as above,  the simply connected K\"ahler manifold $X$ admits
a holomorphic immersion $\varphi: X\ra \CC^2$ or $\varphi: X\ra \HH^2_\RR\times S^2$, 
respectively.
Recall that there is a developing diffeomorphism 
$\psi: X\ra \RR^{4}$.  
As in part (i),  by the uniformization theorem for locally flat conformal structures, there exists $\alpha \in\PO(5,1)$ with 
$\alpha \circ \bar \psi = \bar \varphi$. Since the image of $\alpha \circ \bar \psi$ is 
$S^{4} - \{ p\}$, it cannot occur that $\bar \varphi$ takes values in  $\HH^2_\RR\times S^2  =  S^{4} -S^{1}$.
As in part (i), we conclude that 
$\varphi: X\ra \CC^2$ is a holomorphic isometry.
\medskip 

(iii) $\dim_{\RR}  X = 2$. By Schoen's theorem \cite{SC}, $X$ is simply connected and there exists a conformal diffeomorphism to either $S^{2}$ or $\RR^{2}$. In terms of  uniformization of Riemann surfaces we conclude that $X$ is biholomorphic 
to $S^{2}$ or $\CC$.    
\end{proof}

%
%
%
%
%
%
%
%

\end{document}